\documentclass[bibliography=totoc, 12pt]{amsart}
\usepackage[utf8]{inputenc}
\usepackage[margin=1in]{geometry} 
\usepackage{amsmath,amsthm,amssymb,amsfonts, dsfont, graphicx,mathtools,mathrsfs,tikz,hyperref,physics, stmaryrd, bm, tikz-cd,tkz-euclide,comment, enumerate, appendix, todonotes}
\usepackage{subcaption}
\usepackage[backend=biber,bibencoding=utf8,style=numeric,url=true,
    isbn=false,     
    doi=false,           
    abbreviate=true,            
    giveninits=true, 
    natbib=true,
    hyperref=true]{biblatex}
\usetikzlibrary{decorations.pathreplacing,decorations.markings}

\newcommand{\z}{\mathbb{Z}}
\newcommand{\q}{\mathbb{Q}}
\newcommand{\cx}{\mathbb{C}}

\newcommand{\F}{\mathbb{F}}

\newcommand{\PP}{\mathbb{P}}
\newcommand{\TT}{\mathbb{T}}

\newcommand{\Fq}{\mathbb{F}_q}

\DeclareMathOperator{\supp}{supp}

\DeclareMathOperator{\Gal}{Gal}

\DeclareMathOperator{\GL}{GL}

\DeclareMathOperator{\cha}{char}

\DeclareMathOperator{\rk}{rank}

\def\O{\mathcal{O}}

\def\r{\mathbb{R}}

\theoremstyle{plain}
\newtheorem{theorem}{Theorem}[section]

\newtheorem{lemma}[theorem]{Lemma}

\numberwithin{equation}{section}

\theoremstyle{definition}
\theoremstyle{remark}
\newtheorem*{remark}{Remark}

\title{On a question of Davenport and diagonal cubic forms over $\F_q(t)$}
\author[J. Glas]{Jakob Glas}
\author[L. Hochfilzer]{Leonhard Hochfilzer}
\address{(J.G) IST Austria, Am Campus 1, 3400 Klosterneuburg, Austria.}
\address{(L.H.) Mathematisches Institut, Bunsenstraße 3-5, 37073 Göttingen, Germany}
\email{jakob.glas@ist.ac.at, leonhard.hochfilzer@mathematik.uni-goettingen.de}


\begin{document}
\begin{abstract}
Given a non-singular diagonal cubic hypersurface $X\subset\PP^{n-1}$ over $\F_q(t)$ with $\cha(\F_q)\neq 3$, we show that the number of rational points of height at most $|P|$ is $O(|P|^{3+\varepsilon})$ for $n=6$ and $O(\lvert P \rvert^{2+\varepsilon})$ for $n=4$. In fact, if $n=4$ and $\cha(\F_q) >3$ we prove that the number of rational points away from any rational line contained in $X$ is bounded by $O(|P|^{3/2+\varepsilon})$. From the result in $6$ variables we deduce weak approximation for diagonal cubic hypersurfaces for $n\geq 7$ over $\F_q(t)$ when $\cha(\F_q)>3$ and handle Waring's problem for cubes in $7$ variables over $\F_q(t)$ when $\cha(\F_q)\neq 3$. Our results answer a question of Davenport regarding the number of solutions of bounded height to $x_1^3+x_2^3+x_3^3 = x_4^3+x_5^3+x_6^3$ with $x_i \in \Fq[t]$.
\end{abstract}
\maketitle 
\tableofcontents
\section{Introduction}
Given a non-singular cubic form $F\in K[x_1,\dots, x_n]$ with coefficients in a global field $K$, it is natural to study the distribution of rational points on the hypersurface $X\subset \PP^{n-1}$ defined by $F$. In a quantitative sense, this entails understanding the counting function
\begin{equation}\label{Eq:DefN(P)}
    N(P)=\#\{\bm{x}\in \O^n\colon  |\bm{x}|< |P|, F(\bm{x})=0\},
\end{equation}
where $\O\subset K$ is the ring of integers, $P\in \O$ and $|\cdot|$ is a suitable absolute value on $K$. For $n\geq 5$, one generally expects an asymptotic formula of the form
\begin{equation}\label{Eq:ExpectedAsym}
N(P)\sim c|P|^{n-3}
\end{equation}
as $|P|\to\infty$ for some constant $c\geq 0$. For large values of $n$, this has been successfully achieved using the Hardy--Littlewood circle method. For $K=\q$, the current state of the art is due to Hooley~\cite{Hooley+1988+32+98}, who showed that $n\geq 9$ suffices for~\eqref{Eq:ExpectedAsym} to hold. In fact, conditional on unproved hypotheses about certain Hasse--Weil $L$-functions, in~\cite{hooley2014octonary} he pushed his approach further with the outcome that $n\geq 8$ is enough. For $K=\F_q(t)$, using the fact that the analogous hypotheses are in fact theorems by virtue of Deligne's work~\cite{ deligne1980conjecture}, Browning--Vishe~\cite{CubicHypersurfacesBV} proved unconditionally the asymptotic formula~\eqref{Eq:ExpectedAsym} for $n\geq 8$ and $\cha(K)>3$. However, for small values of $n$, an asymptotic remains largely out of reach. Assuming $F$ to be non-singular and diagonal, which means
\begin{equation}\label{Eq:DefinitionDiagonalCubic}
F(\bm{x})=\sum_{i=1}^nF_ix_i^3, \quad F_i\in \O\setminus\{0\},
\end{equation}
Heath-Brown~\cite{DiagonalCubicHB} has provided an upper bound of the form $N(P)\ll |P|^{3+\varepsilon}$ for $n=6$ and $K=\q$, matching the predicted asymptotic up to a factor of $|P|^\varepsilon$. However, his work relies on deep unproven conjectures about certain Hasse--Weil $L$-functions. 

Our first goal of this work is to prove the analogous result unconditionally for $K=\F_q(t)$. One of the main novelties of our work is that we also obtain results when $\cha(K) = 2$. Usually the circle method breaks down in small characteristic due to a Weyl differencing process. We manage to bypass this issue by applying Poisson summation instead, along with a recursion argument regarding the density of solutions of the dual form $F^*$ of $F$. 

From now on we write $\O = \F_q[t]$ and we work with the absolute value given by $\lvert P \rvert = q^{\deg P}$ for $P \in \O$. By abuse of notation we also write $|\bm{x}|\coloneqq \max_{i}|x_i|$ for $\bm{x}=(x_1,\dots,x_n)\in\O^n$.
\begin{theorem}\label{Th:TheTheorem}
Let $K=\F_q(t)$ with $\cha(K)\neq 3$. Suppose $F$ is given by~\eqref{Eq:DefinitionDiagonalCubic}.
Then for $n=6$ we have 
\[
N(P)\ll |P|^{3+\varepsilon}.
\]
\end{theorem}
In applications of the circle method one frequently uses upper bounds for the counting function  
\[
M(P)=\# \left\{\bm{x}\in \O^6\colon  x_1^3+x_2^3+x_3^3 = x_4^3+x_5^3+x_6^3 \colon \lvert \bm{x} \rvert < \lvert P \rvert \right\} 
\]
to estimate the contribution from the minor arcs. Until now the strongest estimate followed from Hua's lemma, which gives $M(P)\ll |P|^{7/2+\varepsilon}$. In a 1964 letter to Keith Matthews~\cite{davLetter} Davenport asked whether one could achieve the bound $M(P)\ll |P|^{3+\varepsilon}$. Theorem~\ref{Th:TheTheorem} provides an affirmative answer to his question. 

For $n=4$ the situation is more complicated and one does not expect $\eqref{Eq:ExpectedAsym}$ to hold in general. The cubic surface $X\subset\PP^{3}$ might contain rational lines and any such will contribute $\gg |P|^2$ rational points to the counting function $N(P)$. According to Manin's conjecture~\cite{franke1989rational}, one expects
\begin{equation}\label{Eq:Expectn=4}
N^\circ(P)\sim c|P| (\log \lvert P \rvert)^{\rho -1},
\end{equation}
where $N^\circ(P)$ only counts rational points that do not lie on any rational line contained in $X$ and $\rho$ is the rank of the Picard group of $X$. 
 
Over $K=\q$, partial progress was made by Heath-Brown~\cite{DiagonalCubicHB}, who showed how to isolate the contribution to $N(P)$ coming from points on rational lines  when $F$ is diagonal. He also managed to give an upper bound of the form $N^\circ(P)\ll |P|^{3/2+\varepsilon}$, again only conditionally on certain conjectures about Hasse--Weil $L$-functions. As for $n=6$, working over $K = \mathbb{F}_q(t)$ allows us to establish the estimates unconditionally and we also succeed in isolating the contribution coming from points on rational lines under certain restrictions on the characteristic of $K$. 
\begin{theorem}\label{Th:TheoTheorem.n=4}
Suppose $F$ is given by~\eqref{Eq:DefinitionDiagonalCubic}. If $\mathrm{char}(K)>3$, then for $n=4$, we have
\[
N^\circ(P)\ll |P|^{3/2+\varepsilon},
\]
where $N^\circ(P)$ is defined as $N(P)$ with the extra condition that $\bm{x}$ does not lie on any rational line contained in the surface $F=0$. These lines, if they exist, are of the form
\[
b_i x_i + b_j x_j = b_k x_k + b_l x_l = 0,
\]
 for some $b_i,b_j,b_k,b_l \in K$ such that
\[
\left(\frac{b_i}{b_j} \right)^3 = \frac{F_i}{F_j}, \quad \text{and} \quad \left( \frac{b_k}{b_l} \right)^3 = \frac{F_k}{F_l},
\]
where $\{i,j,k,l\} = \{1,2,3,4\}$.

While if $\cha(K)=2$, then for $n=4$ we have 
\[
 N(P)\ll |P|^{2+\varepsilon}.
\]
\end{theorem}

In characteristic 2 the shape of the dual form of $F$ prevents us from isolating the contribution coming from rational points on rational lines to $N(P)$. However, we still manage to give a non-trivial upper bound for the counting function $N(P)$, thereby providing evidence that the main contribution to $N(P)$ comes from points on rational lines.

Our work also shares some similarity with the recent findings of Wang. In~\cite{wang2021approaching} he established an asymptotic formula for $N(P)$ for diagonal cubic forms over $\q$ when $n=6$  conditional on conjectures about mean values of ratios of $L$-functions and the large sieve. His approach required to isolate the contribution coming from rational points on rational linear subspaces, which he achieved in~\cite{wang2021isolating}, similar to Heath-Brown's~\cite{DiagonalCubicHB} treatment when $n=4$. 
 It would be interesting to see to what extent his work can be made unconditional over $\F_q(t)$. 

So far we have ignored the constant $c$ appearing in the asymptotic formula~\eqref{Eq:ExpectedAsym}, despite its arithmetic significance. It encapsulates information about the existence of rational points on $X$ and has received a conjectural interpretation as an adelic volume by Peyre~\cite{peyre1995hauteurs}. For $n\geq 6$ it is expected to be positive as soon as $X(K_\nu) \neq \emptyset$ for all completions $K_\nu$ of $K$, or in other words, it reflects that $X$ is expected to satisfy the Hasse principle.  A key feature of the circle method is that when it provides an asymptotic formula, it automatically confirms the Hasse principle. So in particular, thanks to Hooley~\cite{Hooley+1988+32+98}, we know that the Hasse principle holds for non-singular cubic forms in $n\geq 9$ variables over $\q$ and the work of Browning--Vishe establishes the Hasse principle for non-singular cubic forms over $\F_q(t)$ in at least $8$ variables. 

In fact, by imposing further congruence conditions on $\bm{x}$ in the definition of $N(P)$ in~\eqref{Eq:DefN(P)} Browning--Vishe show that $X$ satisfies weak approximation, which means that under the diagonal embedding
\[
X(K)\longrightarrow \prod_{\nu}X(K_\nu)
\]
the image of $X(K)$ is dense with respect to the product topology. 
Using Theorem~\ref{Th:TheTheorem} as a mean value estimate for the minor arc contribution, we can apply a classical version of the circle method to draw the same conclusions for diagonal cubic forms in $n\geq 7$ variables. 
\begin{theorem}\label{Th:WeakApprox}

Let $K=\F_q(t)$ with $\cha(K)> 3$ and $F$ be a diagonal cubic form in $n\geq 7$ variables. Then the hypersurface $X\subset \PP^{n-1}$ cut out by $F$ satisfies the Hasse principle and weak approximation.

\end{theorem}
One reason for being able to deal with fewer variables than Browning--Vishe is that when $F$ is diagonal we have better control over the exponential sums involved and that we get stronger estimates for the density of solutions of bounded height of the dual form $F^*$ of $F$. However, this alone along with the estimates by Browning--Vishe on averages of exponential sums would not be sufficient to prove Theorem~\ref{Th:TheTheorem}--\ref{Th:WeakApprox}. We additionally make use of slightly better estimates through an argument that enables us to bypass the lack of a convenient form of partial summation over $K$. 

It should be noted that the Hasse principle over $K=\F_q(t)$ is an easy consequence of the Lang--Tsen theory of $C_i$ fields for $n\geq 10$, which in fact establishes that $X(K)\neq \emptyset$ in this case. For smaller values of $n$, only little is known about the Hasse principle or weak approximation over $\F_q(t)$. Colliot-Th\'el\`ene~\cite{colliot2003points} has established the Hasse principle for diagonal cubic forms in $n\geq 5$ variables when $q\equiv 2 \mod 3$ and for $n=4$ for the same range of $q$ under some additional combinatorial constraints on the coefficients of $F$. Furthermore, for arbitrary non-singular cubic hypersurfaces $X\subset \PP^{n-1}$ Tian~\cite{tian2017hasse} has shown that the Hasse principle holds when $\cha(K)>5$ and $n\geq 6$. Assuming the existence of a rational point, Tian--Zhang~\cite{tianWa} have also verified  that $X$ satisfies weak approximation at places of good reduction whose residue fields have at least 11 elements as soon as $n\geq 4$. In fact, the results by Colliot-Th\'el\`ene, Tian and Tian--Zhang were all shown to hold for any global function field $K$ of a smooth curve over a finite field. \\

As a further application of Theorem~\ref{Th:TheTheorem}, we are able to improve Waring's problem over $\F_q(t)$ for cubes.  Waring's problem in degree $d$ in this context is concerned with finding the smallest value of $n$ such that  
\[
P=x_1^d+\cdots +x_n^d
\]
has a solution in $\bm{x}\in \O^n$ for every $P\in\O$ with sufficiently large degree. Over $\F_q(t)$, in contrast to the integer setting, there might be global obstructions for $P$ to be representable as a sum of $d$-th powers, for example if its leading coefficient is not a sum of $n$ $d$-th powers in $\F_q$. Therefore, one usually restricts to $P\in\mathbb{J}_q^d[t]$, which is defined as the additive closure of $d$-th powers in $\F_q[t]$. In order to avoid cancellation in the $x_i$ variables coming from the terms of degree larger than $\deg P$, it is more natural to consider the \emph{strict Waring problem}. There, one is concerned with finding the minimal number $G_q(d) = n$ such that every sufficiently large polynomial $P\in\mathbb{J}_q^d[t]$ can be written as 
\[
P=x_1^d+\cdots +x_n^d,
\]
where $\deg x_i \leq \left\lceil \frac{\deg P}{d} \right\rceil$. In order to study a more refined version of Waring's problem, we introduce the quantity $\widetilde{G}_q(d)$, which is the smallest number $n$ such that we obtain an asymptotic formula for
\[
R_{n}(P)=\#\{\bm{x}\in\O^n\colon |\bm{x}|\leq q^{\left\lceil \frac{\deg(P)}{d}\right\rceil}, \, x_1^d+\cdots + x_n^d=P\},
\]
for $P\in \mathbb{J}_q^d[t]$ as $\deg(P)\to\infty$. 
In his PhD thesis~\cite{Kubota1974} Kubota tackled the asymptotic strict Waring problem over $\F_q(t)$ and showed
$\widetilde{G}_q(d) \leq 2^d+1$ whenever $\cha(\F_q) >d$. The restriction in Kubota's work on the characteristic comes from Weyl differencing, producing a factor of $d!$ and hence rendering trivial bounds when estimating exponential sums if $\cha(\F_q)\leq d$. For degrees $d\geq 4$ this was improved by Liu--Wooley~\cite{LiuWooley2010} by replacing Weyl differencing with an application of the large sieve to also obtain results for $\cha(\F_q)\leq d$.

Returning to the case of cubes, in characteristic $2$ the current state of the art is due to Car--Cherly~\cite{car-cherly} who showed $\widetilde{G}_{2^h}(3) \leq 11$. They managed to avoid Weyl differencing with an application of Poisson summation along with a version of Weyl's inequality in characteristic $2$ developed in~\cite{car1992sommes}. 

Further, work by Gallardo~\cite{gallardo2000restricted} and Car--Gallardo~\cite{car-gallardo} shows
\[
G_q(3) \leq \begin{cases}
7, \quad &\text{if $q \notin \{7,13,16\}$} \\
8, &\text{if $q \in \{13,16\}$} \\
9, &\text{if $q = 7$.}
\end{cases}
\]
Rather than using a circle method approach, the last set of bounds are achieved using elementary arguments. As a result these methods do not produce an asymptotic formula, hence do not yield new bounds for $\widetilde{G}_q(3)$.

We can again use Theorem~\ref{Th:TheTheorem} as a minor arc mean value estimate in order to improve the current best known bound for $\widetilde{G}_q(3)$ for any $q$ not divisible by $3$ as well as for $G_7(3)$, $G_{13}(3)$ and $G_{16}(3)$. Our work on Waring's problem for cubes constitutes a significant improvement on the current state of the art. In particular, our result improves the previously best known upper bound of $\widetilde{G}_q(3)$ by $4$ variables if $q$ is even and by $2$ variables if $q$ is odd.
\begin{theorem}\label{Th:Waring}
If $\cha (\mathbb{F}_q) \neq 3$, then we have $\widetilde{G}_q(3) \leq 7$ and thus also $G_q(3) \leq 7$. 
\end{theorem}
This theorem is the function field counterpart of a result by Hooley~\cite{hooley1986waring}, who proved the asymptotic Waring problem for cubes over integers in $n \geq 7$ variables conditional on hypotheses on certain Hasse--Weil $L$-functions. We also obtain a power saving error term in the asymptotic formula for $R_n(P)$. The best unconditional result in the integer setting is due to Vaughan~\cite{Vaughan86_cubic_waring}, who resolved the asymptotic Waring problem for cubes in $8$ variables, although he obtained only log savings in the error term.

To deduce Theorem~\ref{Th:Waring} from Theorem~\ref{Th:TheTheorem}, we require a power saving when estimating a certain Weyl sum. For Waring's problem this has been carried out by Car~\cite{car1992sommes}, which allows us to establish Theorem~\ref{Th:Waring} in characteristic 2. 
Although it would be possible to adapt the work of Car adequately to handle the Weyl sums appearing in the treatment of weak approximation and thus extend Theorem~\ref{Th:WeakApprox} to the case $\cha(K)=2$, we have decided against including such an adaption here given the length of our paper . 

While the techniques used to prove Theorems~\ref{Th:TheTheorem} --~\ref{Th:Waring} are not applicable when $\cha (K) = 3$, one can almost trivially deal with the problems directly. In fact, studying the solutions to the diagonal cubic equation~\eqref{Eq:DefinitionDiagonalCubic} reduces to solving a system of linear equations. In particular, the Hasse principle and weak approximation hold trivially. Further it is easy to see that $\widetilde{G}_q(3) = 1$ holds when $\cha (K) = 3$.

\subsection*{Outline}
To prove Theorem~\ref{Th:TheTheorem} and Theorem~\ref{Th:TheoTheorem.n=4} we employ a technique known as the \emph{delta method} over $\F_q(t)$ developed by Browning--Vishe~\cite{CubicHypersurfacesBV}, but which is much simpler than the version of Heath-Brown~\cite{DiagonalCubicHB} invoked over the integers. The starting point of the delta method is a smooth decomposition of the Kronecker delta function, a technique that goes back to Duke--Friedlander--Iwaniec~\cite{duke1993bounds}. Over $\F_q(t)$, indicator functions of intervals are smooth in an appropriate sense and so this decomposition is essentially rendered trivial.

In Section~\ref{Se: FFBackground}, we begin by reviewing some essential facts that are required to perform the analysis and arrive at an expression of the form 
\[
    N(w,P)=|P|^n\sum_{\substack{r\text{ monic}\\ |r|\leq \widehat{Q}}}|r|^{-n}\sum_{\bm{c}\in \O^n}S_r(\bm{c})I_r(\bm{c}),
\]
for a weighted version of the main counting function, involving certain exponential sums $S_r(\bm{c})$ and oscillatory integrals $I_r(\bm{c})$.

In Sections~\ref{Se:IntegralEstimates} and~\ref{Se:ExpSums}, we estimate the integrals $I_r(\bm{c})$ and the exponential sums $S_r(\bm{c})$, respectively. More precisely, we obtain cancellations when averaging $S_r(\bm{c})$ over $r$  giving essentially optimal bounds. 
These estimates are possible due to work by Deligne~\cite{deligne1980conjecture} and the required analysis of the relevant $L$-functions has been carried out in~\cite[Section 3]{CubicHypersurfacesBV}. The quality of the estimates of the exponential sums is connected to the dual form of the cubic form. This prompts us to study its rational solutions in Section~\ref{Se: DualForm}.

Classically, to combine these estimates one would use partial summation, a tool that is not available in a useful form to us in the function field setting. In~\cite{CubicHypersurfacesBV} this causes significant difficulty, and in fact the approach by Browning--Vishe comes with a slight loss in the estimates rendering them insufficient for our purposes. We can resolve this issue with Lemma~\ref{Le: I_r(theta,c) deps on |r|}, where we show that $I_r(\bm{c})$ only depends on the absolute value of $r$ and so via $q$-adic summation we can separate the quantities without any loss.

In Section~\ref{Se:ActivationDelta}, we combine the estimates using this new approach and finish our treatment in the case $n=6$, thereby proving Theorem~\ref{Th:TheTheorem}. In the case $\cha(K)=2$, it turns out that the dual form $F^*$ of $F$ is again a non-singular cubic form. For this reason, in Section~\ref{subsec:E_2(P)}, we can introduce a self-improving process in the proof of Theorem~\ref{Th:TheTheorem} and the second part of Theorem~\ref{Th:TheoTheorem.n=4} that turns any saving into the desired upper bound. Finally, we use Theorem~\ref{Th:TheTheorem} as a mean value estimate in an application of the classical circle method to deal with the asymptotic Waring's problem for cubes and weak approximation for diagonal cubic hypersurfaces in $n\geq 7$ variables in Section~\ref{Se:WaringWA}.

If $n=4$ and $\cha(K)>3$ we need to deal separately with the terms coming from \emph{special solutions} of the dual form. This is the content of Section~\ref{Se:SpecialSolutions}, where we show that these terms correspond to points coming from rational lines on $X$. 
\subsection*{Conventions} The letter $\varepsilon$ will always denote an arbitrarily small positive real number, whose value might change from one line to the next. All of the implied constants throughout the paper are allowed to depend on $\varepsilon$, the cardinality of the constant field $q$ and on the form $F$. 

\subsection*{Acknowledgements} The authors would like to thank Tim Browning for suggesting this project. Further they are grateful for his and Damaris Schindler's helpful comments. We would also like to thank Efthymios Sofos for bringing Davenport's question to our attention and Keith Matthews for providing us with scanned copies of the original correspondence.

\section{Function field background}\label{Se: FFBackground}
In this section we collect some basic facts concerning analysis over function fields. A more detailed summary can be found in~\cite[Chapter 5]{Browning2021}. 
Let $K=\F_q(t)$ with ring of integers $\O=\F_q[t]$ and $K_\infty=\F_q((t^{-1}))$ be the field of Laurent series in $t^{-1}$. For $M\in\r$, we shall write $\widehat{M}\coloneqq q^M$. Any $\alpha \in K_\infty\setminus\{0\}$ can be written uniquely as 
\begin{equation}\label{Eq: LaurentSeriesRepresentation}
\alpha = \sum_{i\leq M}\alpha_it^i,\quad \alpha_M\neq 0,
\end{equation}
for some $M\in \z$. If we set $|\alpha|\coloneqq \widehat{M}$, then $|\cdot|$ naturally extends the absolute value induced by $t^{-1}$ on $K$ to $K_\infty$. We also note that $K_\infty$ is the completion of $K$ with respect to this absolute value. The analogue of the unit interval in $K_\infty$ is given by 
\[
\TT\coloneqq \{\alpha\in K_\infty\colon |\alpha|<1\}.
\]
In fact, $K_\infty$ is a local field and thus can be endowed with a unique Haar measure $\dd \alpha $ such that $\int_{\TT}\dd\alpha=1$. We can extend the absolute value to $K^n_\infty$ by $|\bm{\alpha}|=\max_{i=1,\dots, n}|\alpha_i|$ and the Haar measure by $\dd \bm{\alpha}=\dd \alpha_1 \cdots \dd\alpha_n$ for $\bm{\alpha}=(\alpha_1,\dots, \alpha_n)\in K^n_\infty$.\\

Just like over the rational numbers, Dirichlet's approximation Theorem holds. That is, for any $\alpha \in \TT$ and $Q \in \mathbb{N}$ there exist polynomials $a, r \in \O$ with $r$ monic such that $(a,r)=1$ and $|a|<|r|\leq \widehat{Q}$ satisfying
\[
\left\lvert \alpha - \frac{a}{r} \right\rvert < \frac{1}{\lvert r \rvert\widehat{Q}}.
\]
In fact, from the ultrametric property it follows that Dirichlet's approximation Theorem is already enough to obtain for any $Q\geq 1$ an analogue of a \emph{Farey dissection} of the unit interval: 
\begin{equation}\label{Eq: DirichletDissection}
\TT=\bigsqcup_{\substack{|r|\leq \widehat{Q}\\r \text{ monic} }}\bigsqcup_{\substack{|a|<|r|\\(a,r)=1}}\{\alpha\in\TT\colon |r \alpha-a|<\widehat{Q}^{-1}\},
\end{equation}
where $a,r\in \O$.\\

\noindent\textbf{Characters.} For $\alpha \in K_\infty$ given by~\eqref{Eq: LaurentSeriesRepresentation}, we define 
\[
\psi\colon K_\infty\to\cx^\times, \quad \psi(\alpha)=e\left(\frac{\Tr_{\F_q/\F_p}(\alpha_{-1})}{p}\right),
\]
and set $\psi(0)=1$, where as usual we write $e(x)=\exp(2\pi i x)$ for $x \in \mathbb{R}$. It is easy to see that $\psi$ is a non-trivial additive character of $K_\infty$ that satisfies for $x\in K_\infty$ and $N\in\z_{\geq 0}$,
\begin{equation}\label{Eq: OrthogonalityOfCharacters}
    \int_{|\alpha|<\widehat{N}^{-1}}\psi(\alpha x)\dd\alpha=\begin{cases}\widehat{N}^{-1} &\text{if }|x|<\widehat{N},\\ 0 &\text{otherwise.}\end{cases}
\end{equation}
In particular, if $x \in \O$ then this implies
\[
  \int_{\mathbb{T}}\psi(\alpha x)\dd\alpha=\begin{cases} 1 &\text{if }x=0,\\ 0
  &\text{otherwise.}\end{cases}
\]

Further, we will make frequent use of the following formulae for exponential sums. If $r,a \in \O$ are such that $r \neq 0$, then
\[
\frac{1}{\lvert r \rvert} \sum_{\lvert x \rvert < \lvert r \rvert} \psi \left( \frac{ax}{r} \right) = \begin{cases}
1 \quad &\text{if $r \mid a$,} \\
0 &\text{otherwise.}
\end{cases}
\]
We also obtain the expected outcome for Ramanujan sums of prime powers. Let $a, \varpi \in \O$ be such that $\varpi$ is prime and let $k \geq 1$ be a natural number. Then we have
\[
\sideset{}{'}\sum_{\lvert x \rvert < \lvert \varpi \rvert^k} \psi \left( \frac{ax}{\varpi^k} \right) = \begin{cases}
0 \quad &\text{if $\varpi^{k-1} \nmid a$,} \\
-\lvert \varpi \rvert^{k-1} &\text{if $\varpi^{k-1} \parallel a$,} \\
\lvert \varpi \rvert^{k-1} (\lvert \varpi \rvert -1) &\text{if $\varpi^k \mid a$,}
\end{cases}
\]
where the notation $\sideset{}{'}\sum_{\lvert x \rvert < \lvert \varpi \rvert^k}$ indicates that the sum runs over $x$ which are coprime to $\varpi$.

\noindent\textbf{Poisson Summation.} We call a function $w \colon K_\infty^n \rightarrow \mathbb{C}$ \emph{smooth} if it is locally constant. Denote by $S(K_\infty^n)$ the space of all smooth functions $w \colon K_\infty^n \rightarrow \mathbb{C}$ with compact support. If  $w \in S(K_\infty^n)$ then we call $w$ a  \emph{Schwarz-Bruhat function}. For such functions the Poisson summation formula~\cite[Lemma 2.1]{CubicHypersurfacesBV} holds. 
\begin{lemma}
    Let $f \in K_\infty[x_1, \hdots, x_n]$ and let $w \in S(K_\infty^n)$. Then  we have
    \begin{equation} \label{eq.poisson summation}
        \sum_{\bm{z} \in \O^n} w(\bm{z}) \psi(f(\bm{z})) = \sum_{\bm{c} \in \O^n} \int_{K_\infty^n} w(\bm{u}) \psi(f(\bm{u}) + \bm{c} \cdot \bm{u}) \dd \bm{u}.
    \end{equation}
\end{lemma}
\noindent\textbf{Delta method.} Given a polynomial $F\in\O[x_1,\dots,x_n]$ and $w\in S(K^n_\infty)$, we are interested in the counting function 
\begin{equation*}
    N(w,P)=\sum_{\substack{\bm{x}\in \O^n\\ F(\bm{x})=0}}w\left(\frac{\bm{x}}{P}\right).
\end{equation*}
For estimating the integrals appearing in our work, it is necessary to work with such a weighted counting function, since we require $\nabla F$ to be bounded away from 0 on $\supp(w)$. To estimate our original counting function defined in~\eqref{Eq:DefN(P)}, it suffices to take $w$ to be the characteristic function of the set $\{\bm{x} \in \TT\colon |\bm{x}|=q^{-1}\}$. Indeed, it follows that
\[
N(w,P)=\# \{\bm{x}\in \O^n\colon F(\bm{x})=0, |\bm{x}|=q^{-1}|P|\},
\]
so that an upper bound of the shape $N(P,w) \ll \lvert P\rvert^{k}$
yields $N(P) \ll \lvert P\rvert^{k+\varepsilon}$ for any $\varepsilon > 0$ by summing over $q$-adic ranges for $|P|$. \\

For a fixed parameter $Q\geq 1$ to be specified later, we deduce from~\eqref{Eq: DirichletDissection} and ~\eqref{Eq: OrthogonalityOfCharacters} the identity
\[
N(w,P)=\sum_{\substack{r\text{ monic}\\ |r|\leq \widehat{Q}}}\sideset{}{'}\sum_{|a|<|r|}\int_{|\theta|<|r|^{-1}\widehat{Q}^{-1}}S(a/r+\theta)\dd \theta,
\]
where $\sum'_{|a|<|r|} $ means that we sum over $a\in\O$ with $(a,r)=1$ only and
\[
S(\alpha)=\sum_{\substack{\bm{x}\in\O^n}}\psi(\alpha F(\bm{x}))w(\bm{x}/P)
\]
for $\alpha \in \TT$. As explained in~\cite[Chapter 4]{CubicHypersurfacesBV}, since $w$ is a Schwartz-Bruhat function we can evaluate $S(\theta+a/r)$ using Poisson summation~\eqref{eq.poisson summation} to obtain
\begin{equation}\label{Eq: DeltaMethod}
    N(w,P)=|P|^n\sum_{\substack{r\text{ monic}\\ |r|\leq \widehat{Q}}}|r|^{-n}\int_{|\theta|<|r|^{-1}\widehat{Q}^{-1}}\sum_{\bm{c}\in \O^n}S_r(\bm{c})I_r(\theta,\bm{c})\dd\theta,
\end{equation}
where 
\begin{equation}\label{Eq: Definition S_r(c)}
    S_r(\bm{c})=\sideset{}{'}\sum_{|a|<|r|}\sum_{|\bm{x}|<|r|}\psi\left(\frac{aF(\bm{x})+\bm{c}\cdot \bm{x}}{r}\right)
\end{equation}
and 
\begin{equation}\label{Eq: Definition I_r(theta,c)}
    I_r(\theta,\bm{c})=\int_{K_\infty^n}w(\bm{x})\psi\left(\theta P^3F(\bm{x})+\frac{P \bm{c}\cdot \bm{x}}{r}\right)\dd \bm{x}.
\end{equation}
The expression~\eqref{Eq: DeltaMethod} is the starting point for our work and from now on we will mostly be concerned about estimating the integrals $I_r(\theta,\bm{c})$ and the sums $S_r(\bm{c})$.

\section{Integral estimates}\label{Se:IntegralEstimates}
As a preliminary lemma we note the following result on a linear change of variables, the proof of which is completely analogous to the proof of Lemma 7.4.2 in~\cite{igusa2007introduction}.
 \begin{lemma} \label{lem.change_of_variables}
 Let $R_1, \hdots, R_n \in \r$ and let $\Gamma \subset K_\infty^n$ be the region given by
 \[
 \Gamma = \{ \bm{x} \in K_\infty^n \colon \lvert x_i \rvert \leq \widehat{R_i} \}.
 \]
 Let $g \colon \Gamma \rightarrow \cx$ be a continuous function and let $M \in \GL_n(K_\infty)$. Then we have
 \[
 \int_{\Gamma} g(\bm{x}) \dd \bm{x}= \lvert \det M \rvert \int_{M \bm{\alpha} \in \Gamma} g(M\bm{\beta}) \dd \bm{\beta}.
 \]
 \end{lemma}
For $f\in K_\infty[x_1,\dots,x_n]$, we denote by $H_f$ its height, that is, the maximum of the absolute values of its coefficients. Given $\gamma\in K_\infty$, $\bm{w}\in K_\infty^n$ and $f\in K_\infty[x_1,\dots, x_n]$,  integrals of the form
\begin{equation}
    J_f(\gamma,\bm{w})\coloneqq \int_{K_\infty^n} w(\bm{x})\psi(\gamma f(\bm{x})+\bm{w}\cdot\bm{x})\dd \bm{x}
\end{equation}
appear quite frequently in our work. We shall now collect the required estimates for them. Upon noting that $w(\bm{x})=\chi_{\TT}(\bm{x})-\chi_{t^{-1}\TT}(\bm{x})$, the next lemma follows directly from~\cite[Lemma 2.4]{CubicHypersurfacesBV}.
\begin{lemma}
\label{lem.J_f_vanishing}
Let $\gamma \in K_\infty$ and $\bm{w} \in K_\infty^n$ be such that $\lvert \bm{w} \rvert > q$ and $\lvert \bm{w} \rvert \geq H_f \lvert \gamma \rvert$. Then $J_f(\gamma,\bm{w}) = 0$.
 \end{lemma}

The next result~\cite[Lemma 2.7]{CubicHypersurfacesBV} is the main ingredient for estimating the integrals $J_f(\gamma,\bm{w})$.

\begin{lemma} \label{lem.integral_vanishes_BV}
    We have
    \[
    \int_{\TT^n \setminus \Omega} \psi (\gamma f(\bm{x}) + \bm{w} \cdot \bm{x}) \dd \bm{x} = 0,
    \]
    where $\Omega \subset \TT^n$ is given by
    \[
    \Omega = \left\{ \bm{x}\in \TT^n \colon \lvert \gamma \nabla f(\bm{x}) + \bm{w} \rvert \leq H_f \max \left\{ 1, \lvert \gamma \rvert^{1/2} \right\} \right\}.
    \]
    
\end{lemma}
In our setting, this leads to the following estimate.
\begin{lemma} \label{lem.J_F_estimate}
Suppose $F\in K_\infty[x_1,\dots, x_n]$ is a non-singular cubic form. Let $\gamma \in K_\infty$ and $\bm{w} \in K_\infty^n\setminus\{\bm{0}\}$ be such that $\lvert \bm{w} \rvert \gg 1$. Then $J_F(\gamma,\bm{w})=0$, unless 
\[
\lvert \bm{w} \rvert \ll \lvert \gamma \rvert \ll \lvert \bm{w} \rvert,
\]
in which case
\[
J_F(\gamma,\bm{w}) \ll \mathrm{meas} ( \{\bm{x} \in \mathrm{supp}(w) \colon \lvert \gamma \nabla F(\bm{x}) + \bm{w} \rvert \ll \lvert \bm{w} \rvert^{1/2} \}).
\]
\end{lemma}
\begin{proof}
First note  $J_F(\gamma,\bm{w}) = 0$ if $\lvert \bm{w}\rvert > \max\{q, H_F \lvert \gamma \rvert \}$ by Lemma~\ref{lem.J_f_vanishing}. Since by assumption $1\ll |\bm{w}|$, we may thus assume $1\ll |\bm{w}|\ll |\gamma|$. For $\bm{a}\in\F_q^n\setminus\{\bm{0}\}$, let 
\[
w_{\bm{a}}(\bm{x})=\begin{cases} 1 &\text{if }|\bm{x}-\bm{a}t^{-1}|<|t|^{-1},\\
0&\text{else.}\end{cases}
\]
We can then write $w(\bm{x})=\sum_{\bm{a}\in\F_q^n\setminus\{\bm{0}\}}w_{\bm{a}}(\bm{x})$, so that
\begin{align}
\begin{split}\label{eq.J_F_indicator_boxes_swap}
J_F(\gamma,\bm{w})&=\sum_{\bm{a}\in\F_q^n\setminus\{\bm{0}\}}\int_{\TT^n}w_{\bm{a}}(\bm{x})\psi(\gamma F(\bm{x})+\bm{w}\cdot \bm{x})\dd\bm{x}\\
&= \sum_{\bm{a}\in\F_q^n\setminus\{\bm{0}\}}q^{-n}\psi(t^{-1}\bm{w}\cdot \bm{a})\int_{\TT^n}\psi(\gamma G_{\bm{a}}(\bm{y})+t^{-1}\bm{w}\cdot \bm{y})\dd \bm{y},
\end{split}
\end{align}
where we performed the change of variables $\bm{y}=t\bm{x}-\bm{a}$ and wrote $G_{\bm{a}}(\bm{y})=F((\bm{y}+\bm{a})t^{-1})$.
From Lemma~\ref{lem.integral_vanishes_BV} we deduce that each inner integral is bounded by 
\[
\text{meas}(\{\bm{y}\in\TT^n\colon |\gamma \nabla G_{\bm{a}}(\bm{y})+t^{-1}\bm{w}|\ll H_{G_{\bm{a}}}|\gamma|^{1/2}\}),
\]
which in turn may be bounded from above by
\begin{equation} \label{Eq:measureIntegral}
  \text{meas}(\{\bm{x}\in \mathrm{supp}(w_{\bm{a}})\colon |\gamma \nabla F(\bm{x})+\bm{w}|\ll H_F|\gamma|^{1/2}\}),
\end{equation}
since $H_{G_{\bm{a}}}\leq H_F$. Denote the set in~\eqref{Eq:measureIntegral} by $\Omega_{\bm{a}}$. Note that since $F$ is assumed to be non-singular, we have $\nabla F(\bm{x})\neq 0$ for all $\bm{x}\in\Omega_{\bm{a}}$. Since 
$\mathrm{supp} (w_{\bm{a}})$
is compact for every $\bm{a}$, this implies $\nabla F(\bm{x})\gg_w 1$ for all $\bm{x}\in\Omega_{\bm{a}}$. 
In particular, unless $\lvert \bm{w} \rvert \gg |\gamma\nabla F(\bm{x})| \gg \lvert \gamma \rvert$ the sets $\Omega_{\bm{a}}$ are all empty and the integral vanishes. Finally the Lemma follows upon noting 
\[
 \mathrm{meas} ( \Omega_{\bm{a}} )\ll \mathrm{meas} ( \{\bm{x} \in \mathrm{supp}(w) \colon \lvert \gamma \nabla F(\bm{x}) + \bm{w} \rvert \ll \lvert \bm{w} \rvert^{1/2} \}),
\]
for any $\bm{a} \in \Fq^n \setminus \{ \bm{0} \}$ and substituting this into~\eqref{eq.J_F_indicator_boxes_swap}.
\end{proof}
Since we work with a diagonal cubic form $F(\bm{x})=\sum_{i=1}^n F_ix_i^3$ with $F_i\in\O\setminus\{0\}$, we have $\nabla F(\bm{x}) = (3F_1x_1^2, \hdots, 3F_nx_n^2)$. Therefore in order to find an upper bound for $J_F(\gamma,\bm{w})$ the following lemma will be useful.
\begin{lemma} \label{lem.measure_parabola}
Let $a,b \in K_\infty$ and consider the set 
\[
P_{a,b} =
\{ x \in \TT \colon \lvert x^2-a \rvert < \lvert b \rvert \}.
\]
Then we have
\[
\mathrm{meas} (P_{a,b}) \ll \min \{\lvert b \rvert^{1/2}, \lvert b \rvert \lvert a \rvert^{-1/2} \}.
\]
\end{lemma}
\begin{proof}
Note first that the result is trivial if $a = 0$ or $b = 0$. Hence we may write
\[
a = \sum_{i \leq K } a_it^i, \quad \text{and} \quad b = \sum_{j \leq M} b_jt^j, 
\]
where $a_K, b_M \neq 0$. We will proceed in two cases.

\noindent
\textbf{Case 1:} $\lvert a \rvert < \lvert b \rvert$. Then via the ultrametric triangle inequality we note
\[
\lvert x^2-a \rvert < \lvert b \rvert \iff \lvert x \rvert^2 < \lvert b \rvert,
\]
for any $x \in \TT$. Thus $\mathrm{meas} (P_{a,b}) \ll \lvert b \rvert^{1/2} = \min \{\lvert b \rvert^{1/2}, \lvert b \rvert \lvert a \rvert^{-1/2} \}$.

\noindent\textbf{Case 2:} $\lvert a \rvert \geq \lvert b \rvert$. Let $x = \sum_{i \leq -1} x_it^i \in \TT$. Then $\lvert x^2-a \rvert < \lvert b \rvert$ can only hold if $\lvert x \rvert^2 = \lvert a \rvert$. In particular $K$ must be even, $K \leq -1$ must hold and $x_{K/2+1} =  \cdots = x_{-1} = 0$. Write
\[
x^2 = \sum_{\ell\leq K} X_\ell t^\ell,
\]
where $X_\ell = \sum_{i+j = \ell} x_i x_j$. Then, requiring
\[
\lvert x^2-a \rvert < \lvert b \rvert = q^M
\]
implies $X_\ell = a_\ell$ for $\ell = M, \hdots, K$. Now $X_K = x_{K/2}^2$, so the condition $X_K = a_K$ yields at most two possible solutions for $x_{K/2}$. Further, since
\[
X_{K-r} =2x_{K/2} x_{K/2-r} +\sum_{\substack{i+j = K-r \\ K/2-r < i,j < K/2}} x_ix_j,
\]
 we find inductively that a solution to  $x_{K/2}^2 = a_K$ uniquely determines $x_{K/2-r}$ for $r = 1, \hdots, M+K$. To summarise, in this case, there are at most two possibilities for the values of the coefficients $x_{-1}, \hdots, x_{M-K/2}$. Therefore we obtain
\[
\mathrm{meas}(P_{a,b}) \ll \mathrm{meas} \left( t^{M-K/2}\TT \right) = q^{M-K/2} = \lvert b \rvert \lvert a \rvert^{-1/2}.
\]
Finally, noticing that $\lvert b \rvert \lvert a \rvert^{-1/2} \leq \lvert b \rvert^{1/2}$ if $\lvert a \rvert \geq \lvert b \rvert$ finishes the proof of this lemma.
\end{proof}

In light of Lemma~\ref{lem.measure_parabola} we thus find
\[
 \mathrm{meas} ( \{\bm{x} \in \mathrm{supp}(w) \colon \lvert \gamma \nabla F(\bm{x}) + \bm{w} \rvert \ll  \lvert \bm{w} \rvert^{1/2} \}) \ll \prod_{i=1}^n \min \{ \lvert \bm{w}\rvert^{-1/4} , \lvert w_i \rvert^{-1/2} \}
\]
if $F$ is a diagonal cubic form. Noting that the expression on the right hand side is $\gg_q 1$ if $\lvert \bm{w}\rvert \ll 1$ we infer from Lemma~\ref{lem.J_F_estimate} 
\begin{equation}\label{Eq: J_F.integral.estimate}
J_F(\gamma,\bm{w}) \ll \prod_{i=1}^n \min \{ \lvert \bm{w}\rvert^{-1/4} , \lvert w_i \rvert^{-1/2} \},
\end{equation}
for all $\gamma \in K_\infty$ and all $\bm{w} \in K_\infty^n\setminus\{\bm{0}\}$.

We will also have to deal with averages of $I_r(\theta,\bm{c})$ over $\theta$, which are of the form
\[
I_r(\bm{c})\coloneqq \int_{|\theta|<|r|^{-1}\widehat{Q}^{-1}}I_r(\theta,\bm{c})\dd \theta.
\]
While we do not have a convenient form of partial summation available in the function field setting, the next lemma will be crucial in replacing this tool.
\begin{lemma}\label{Le: I_r(theta,c) deps on |r|}
Let $r_1,r_2\in \O$ be such that $|r_1|=|r_2|$. Then $I_{r_1}(\bm{c})=I_{r_2}(\bm{c})$.
\end{lemma}
\begin{proof} Write $r=r_1$ for brevity. We shall show that $I_r(\bm{c})$ only depends on the absolute value of $r$. Indeed, recalling~\eqref{Eq: Definition I_r(theta,c)}, for $\bm{c}$ fixed we have
\begin{align}
    I_r(\bm{c})&=\int_{|\theta|<|r|^{-1}\widehat{Q}^{-1}}\int_{K^n_\infty} w(\bm{x}) \psi\left(\theta P^3 f(\bm{x})+\frac{P\bm{c}\cdot \bm{x}}{r}\right)\dd\bm{x}\dd\theta\nonumber\\
    &= |r|^{n}\int_{K_\infty^n}w(r\bm{y})\psi(P\bm{c}\cdot \bm{y})\int_{|\theta|<|r|^{-1}\widehat{Q}^{-1}}\psi(\theta P^3 r^3 f(\bm{y}))\dd\theta\dd\bm{y}\label{Eq: IndependenceInt},
\end{align}
where we used Fubini's theorem and applied the change of variables $\bm{y}=\bm{x}r^{-1}$. It follows from~\eqref{Eq: OrthogonalityOfCharacters} that 
\[
\int_{|\theta|<|r|^{-1}\widehat{Q}^{-1}}\psi(\theta P^3r^3f(\bm{y}))\dd\theta =\begin{cases} (|r|\widehat{Q})^{-1} &\text{if }|P^3f(\bm{y})|<|r|^{-2}\widehat{Q},\\ 0 &\text{else.}\end{cases}
\]
We conclude that the value of the inner integral in~\eqref{Eq: IndependenceInt} only depends on $|r|$ for $\bm{y}$ and $\bm{c}$ fixed. The claim now follows, since $w$ only depends on the absolute value of its argument. 
\end{proof}
To highlight this dependence, we shall write $I_{\widehat{Y}}(\bm{c})=I_r(\bm{c})$ if $|r|=\widehat{Y}$ from now on. In the notation above, for $r \in \mathcal{O}\setminus\{0\}$, $\bm{c} \in \mathcal{O}^n$, $\theta \in \TT$ and $P \in \O$ we have
\[
I_r(\theta,\bm{c}) = J_F\left(P^3 \theta, \frac{P}{r} \bm{c} \right).
\]
Since $I_r(\theta,\bm{c})$ vanishes unless $\frac{|P||\bm{c}|}{|r|}\ll |\theta||P|^3\ll \frac{|P||\bm{c}|}{|r|}$, we deduce from~\eqref{Eq: J_F.integral.estimate} the following integral estimate.
\begin{lemma} \label{lem.integral_estimate_for_I_r}
Let $Y \geq 0$, $\bm{c} \in \mathcal{O}^n\setminus\{\bm{0}\}$,  and  $P \in \O$. Then 
\begin{equation*}
    I_{\widehat{Y}}(\bm{c})\ll \min\left\{\frac{|\bm{c}|}{\widehat{Y}|P|^2},\widehat{Y}^{-1}\widehat{Q}^{-1}\right\} \prod_{i = 1}^n \min \left\{ \left( \frac{\lvert P \rvert \lvert \bm{c}\rvert}{\widehat{Y}} \right)^{-1/4}, \left( \frac{\lvert P \rvert \lvert c_i \rvert}{\widehat{Y}} \right)^{-1/2} \right\}.
\end{equation*}
\end{lemma}
So far we have not yet achieved any non-trivial estimates for $I_{\widehat{Y}}(\bm{0})$ and in fact we will have to do slightly better than the trivial bound for our treatment. 
\begin{lemma} \label{lem.I(0) estimate}
Assume $n \geq 4$. Let $P \in \mathcal{O} \setminus \{ \bm{0} \}$. Then for any  $Y \geq 1$ we have
\begin{equation*}
    I_{\widehat{Y}}(\bm{0}) \ll \lvert P \rvert^{-3 + \varepsilon}.
\end{equation*}
\end{lemma}
\begin{proof}
For $r \in \mathcal{O} \setminus \{ 0 \}$ such that $\lvert r \rvert = {\widehat{Y}}$, Lemma~\ref{lem.integral_vanishes_BV} gives
\[
\tilde{I}_r(\theta,\bm{0}) \coloneqq \int_{\TT^n} \psi \left( \theta P^3 F(\bm{x})\right) \dd \bm{x}  \ll \mathrm{meas} (\{ \bm{x} \in \TT^n \colon \lvert \nabla F(\bm{x}) \rvert \leq \max\{1, \lvert \theta \rvert \lvert P \rvert^{3}\}^{-1/2}\}).
\]
Now it is not hard to see that $I_r(\theta,\bm{0}) = \tilde{I}_r(\theta,\bm{0}) - q^{-n} \tilde{I}_r(q^{-3}\theta,\bm{0})$. From Lemma~\ref{lem.J_F_estimate} we deduce
\[
I_r(\theta,\bm{0}) \ll \mathrm{meas}( \{ \bm{x} \in \TT^n \colon \lvert \nabla F(\bm{x}) \rvert \ll  \max\{1, \lvert \theta \rvert \lvert P \rvert^{3}\}^{-1/2}\}).
\]
Since $F$ is diagonal we have $\lvert \nabla F(\bm{x}) \rvert \geq \lvert \bm{x} \rvert^2$ whence 
\[
I_r(\theta,\bm{0}) \ll \max\{1, \lvert \theta \rvert \lvert P \rvert^{3}\}^{-n/4}.
\]
By definition of $I_{\widehat{Y}}(\bm{0})$ we may divide the area of integration up as follows
\[
I_{\widehat{Y}}(\bm{0}) = \int_{\lvert \theta \rvert \ll \lvert P \rvert^{-3}} I_r(\theta,\bm{0}) \dd \theta + \int_{ \lvert P \rvert^{-3} \ll \lvert \theta \rvert < \widehat{Q}^{-1} {\widehat{Y}}^{-1}} I_r(\theta,\bm{0}) \dd \theta.
\]
The first term is trivially $O(\lvert P \rvert^{-3})$. For the second term note
\[
\int_{ \lvert P \rvert^{-3} \ll \lvert \theta \rvert < \widehat{Q}^{-1} {\widehat{Y}}^{-1}} I_r(\theta,\bm{0}) \dd \theta \ll \int_{\lvert P \rvert^{-3} \ll \lvert \theta \rvert < \widehat{Q}^{-1} {\widehat{Y}}^{-1}} \lvert P \rvert^{-3n/4} \lvert \theta \rvert^{-n/4} \dd \theta \ll \lvert P \rvert^{-3+\varepsilon}.
\]

The result now follows.
\end{proof}

\section{Exponential sum estimates}\label{Se:ExpSums}
We want to estimate the sum
\begin{align}\label{Eq: Expsum_as_product}
\begin{split}
    S_r(\bm{c})&= \sideset{}{'}\sum_{|a|<|r|}\sum_{|\bm{x}|<|r|}\psi\left(\frac{aF(\bm{x})+\bm{c}\cdot\bm{x}}{r}\right)\\
    &=\sideset{}{'}\sum_{|a|<|r|}\prod_{i=1}^n\sum_{|x|<|r|}\psi\left(\frac{aF_ix^3+c_ix}{r}\right),
    \end{split}
\end{align}
where $F(\bm{x})=\sum_{i=1}^nF_ix_i^3$. The corresponding sum over the integers has already been subject to thorough investigation by Heath-Brown~\cite{heath1983cubic} and Hooley~\cite{hooley1986waring}. Browning--Vishe~\cite{CubicHypersurfacesBV} have translated many of the properties to the function field setting, some of which we shall record here. 

The quality of our estimates is intimately connected to the dual form $F^*$ of $F$, which is an absolutely irreducible polynomial of degree $2^{n-2}\times 3$ whose zero locus parameterises hyperplanes that have a singular intersection with the projective hypersurface cut out by $F$. As explained by Wang~\cite[Appendix D]{wang2021diagonal}, if $F$ is diagonal and $\cha(K)>3$, we can take 
\begin{equation}\label{Def: DualForm}
    F^*(\bm{c})=\left(\prod_{i=1}^nF_i\right)^{2^{n-2}}\prod\left((F_1^{-1}c_1^3)^{1/2}\pm \cdots \pm (F_n^{-1}c_n^3)^{1/2}\right),
\end{equation}
where the inner product runs through all possible combinations of $\pm$. In fact, in~\cite{wang2021diagonal} this is only shown for $K=\q$, but one can check that the requirement $\cha(K)>3$ is sufficient for~\eqref{Def: DualForm} to hold. In characteristic 2, we have the following result.
\begin{lemma}\label{Le: DualChar2}
    Let $K$ be a field of characteristic 2 and $F(\bm{x})=\sum_{i=1}^n F_ix_i^3 \in K[x_1,\dots,x_n]$ be a non-singular cubic form. Then the dual form of $F$ is given by
    \[
    F^*(\bm{c})=\left(\prod_{i=1}^n F_i\right)\sum_{i=1}^n F_i^{-1}c_i^3.
    \]
\end{lemma}
\begin{proof}
By definition the zero locus $V(F^*)\subset \PP^{n-1}$ parameterises points $\bm{c}\in \PP^{n-1}$ such that the hyperplane $\bm{c}\cdot \bm{x}=0$ has a singular intersection with $V(F^*)$. This means, that there exists $\bm{x}\in \PP^{n-1}(\overline{K})$ such that 
\begin{equation}\label{Eq: DualFormChar2}
\rk \begin{pmatrix}
    \nabla F (\bm{x}) \\ \bm{c}\end{pmatrix} = 1, \quad \bm{c}\cdot \bm{x}=0 \quad\text{and}\quad F(\bm{x})=0.
\end{equation}
Since we assume $F$ to be non-singular, the rank condition implies that $\bm{c}$ is proportional to $\nabla F(\bm{x})$, that is, $x_i^2 = \lambda F_i^{-1}c_i$ for some $\lambda \in \overline{K}^\times$ and $i=1,\dots, n$. Any pair $(\bm{x},\bm{c})$ having this property then satisfies $F(\bm{x})=0$ if and only if $\bm{c}\cdot \bm{x}=0$. Moreover, the third condition in~\eqref{Eq: DualFormChar2} is equivalent to 
\[
\sum_{i=1}^n F_i^{-1/2}c_i^{3/2}=0,
\]
where we used that every element of $\overline{K}$ has a unique square-root as $\cha(K)=2$. However, again since we are in characteristic 2, this is is equivalent to 
\[
\sum_{i=1}^nF_i^{-1}c_i^3=0.
\]
The result now follows after clearing denominators.
\end{proof}
Note that if $r_1, r_2\in \O$ are coprime, then 
\begin{equation}\label{Eq: S_r(c) multiplicative}
    S_{r_1r_2}(\bm{c})=S_{r_1}(\bm{c})S_{r_2}(\bm{c}),
\end{equation}
which follows readily from the Chinese remainder theorem.
This essentially reduces the task of estimating $S_r(\bm{c})$ to prime power moduli. Indeed, suppose $S_{\varpi^k}(\bm{c})\leq C |\varpi|^{k\alpha}$ for some $\alpha>0$ and some absolute constant $C$. Let $\Omega(r)$ be the number of prime divisors of $r$. Then by multiplicativity of $S_r(\bm{c})$ we have 
\begin{align*}
    S_r(\bm{c})&=\prod_{\varpi^k\parallel r}S_{\varpi^k}(\bm{c})
    \leq \prod_{\varpi^k\parallel r}C|\varpi|^{k\alpha}
    = C^{\Omega(r)}|r|^{\alpha}
    \ll \tau(r) |r|^\alpha 
    \ll |r|^{\alpha+\varepsilon}
\end{align*}
by the usual estimate for the divisor function $\tau(r)$, see~\cite[Lemma 5.9]{Browning2021}.

Further, if $\varpi$ is irreducible such that $\varpi\nmid F^*(\bm{c})$, then Browning--Vishe~\cite[Section 5]{CubicHypersurfacesBV} show
\begin{equation}\label{Eq: S_w^k(c) vanishes}
    S_{\varpi^k}(\bm{c})=0\quad \text{for }k\geq 2.
\end{equation}

\subsection{Square-free moduli contribution}
Deligne's resolution of the Weil  conjectures~\cite{deligne1974conjecture} shows that we get square-root cancellation for the sums $S_{\varpi}(\bm{c})$ whenever $\varpi$ is suitably generic:
\begin{equation}\label{Eq: DeligneS_r(c)}
    S_\varpi(\bm{c})\ll |\varpi|^{(n+1)/2}|(\varpi, \nabla F^*(\bm{c}))|^{1/2}.
\end{equation}
However, this is not sufficient for our purposes. In the integer setting Hooley~\cite{hooley1986waring} was the first to achieve an extra saving when averaging the sums $S_r(\bm{c})$ over $r$ by appealing to certain hypotheses about Hasse--Weil $L$-functions associated to cubic threefolds. By virtue of Deligne's proof of the Weil conjectures~\cite{deligne1980conjecture}  these hypotheses are in fact theorems in the function field setting. This enabled Browning--Vishe~\cite[Lemma 8.5]{CubicHypersurfacesBV} to establish the following result unconditionally.
\begin{lemma}\label{Le: Hasse-WeilExpBound}
Suppose $n$ is even and $F^*(\bm{c})\neq 0$. Then for any $Z\geq 0$ and $\varepsilon>0$, we have 
\begin{equation*}
    \sum_{\substack{|r|\leq \widehat{Z}\\ (r, \Delta_F F^*(\bm{c}))=1}}\frac{S_r(\bm{c})}{|r|^{(n+1)/2}}\ll |\bm{c}|^{\varepsilon}\widehat{Z}^{1/2+\varepsilon},
\end{equation*}
where $\Delta_F$ is the discriminant of $F$ and by virtue of~\eqref{Eq: S_w^k(c) vanishes} $r$ ranges over square-free values only.
\end{lemma}
\begin{remark}
In fact Browning--Vishe have to consider averages of $S_r(\bm{c})$ twisted by a Dirichlet character of $K_\infty$ since they were unable to separate the integral $I_r(\theta, \bm{c})$ from summation. However, we can resolve this issue with Lemma~\ref{Le: I_r(theta,c) deps on |r|} allowing us to combine Lemma~\ref{Le: Hasse-WeilExpBound} with the integral bounds from Lemma~\ref{lem.integral_estimate_for_I_r} more efficiently.
\end{remark}
\subsection{Pointwise estimates}
For $B\in\O$ fixed and $a,r\in\O\setminus\{0\}$ with $(a,r)=1$, let
\[
S_r(a,c)=\sum_{|x|<|r|}\psi\left(\frac{aBx^3+cx}{r}\right).
\]
In view of~\eqref{Eq: Expsum_as_product} upper bounds for $S_r(a,c)$ directly transform into estimates for $S_r(\bm{c})$. Moreover, by~\eqref{Eq: S_r(c) multiplicative} it suffices to consider the case $r=\varpi^k$, where $\varpi$ is irreducible. Hooley~\cite{hooley1986waring} has provided upper bounds for the integer-analogue of the sum $S_{\varpi^k}(a,c)$ whenever $\varpi\nmid B$. As explained by Heath-Brown~\cite{DiagonalCubicHB}, these estimates also hold if $\varpi \mid B$ when we allow the implied constant to depend on $B$. Hooley's and Heath-Brown's proofs of these results go through almost verbatim in the function field setting and so we spare the reader from the tedious exercise of reproducing them here. To state the final outcome, we need some notation. First, we set $\{\varpi^k,c\}=(\varpi^k,c)$ for $k=2$ and for $k\geq 3$, we define $\{\varpi^k,c\}=|\varpi|^{-1}$ if $\varpi \parallel c$ and $\{\varpi^k,c\}=(\varpi^k,c)$ else. For later use, we generalise this to square-full $r$ by setting
\[
\{r,c\}\coloneqq \prod_{\varpi^k\parallel r}\{\varpi^k,c\}.
\]
We then have 
\begin{align}
    S_{\varpi^k}(a,c)&\ll |\varpi|^{k/2}\lvert\{\varpi^k,c\}\rvert^{1/4} \quad\text{for }k\geq 2.\label{Eq: primePowerEstimates}
\end{align}
We shall also use an estimate of Hua~\cite[Lemma 1.1]{hua_1965}, whose proof, again, readily translates to the function field setting. If $g( {x}) = \sum_{i=0}^d g_ix^i \in \mathcal{O}[x]$, then for any $\varpi \in \mathcal{O}$ irreducible we have
\begin{equation}\label{Eq: HuaEstimate}
    \sum_{\lvert x \rvert < \lvert \varpi \rvert^k} \psi\left( \frac{g(x)}{\varpi^k} \right) \ll |\varpi|^{k(1-1/d)} \lvert ( \varpi^k, g_0, \hdots, g_d) \rvert^{1/d},
\end{equation}
where the constant depends only on $\varepsilon$ and $d$. Originally this was stated in the case when $\varpi \nmid (g_0, \hdots, g_d)$, but the factor $\lvert (\varpi^k, g_0, \hdots, g_d) \rvert^{1/d}$ in the estimate accounts for the possibility of $\varpi \mid (g_0, \hdots, g_d)$. Therefore we obtain
\begin{equation*}
    S_{\varpi^k}(a,c) \ll \lvert  \varpi \rvert^{2k/3},
\end{equation*}
where the implied constant depends on $\varepsilon$ but crucially not on $a$ since we assumed $\varpi \nmid a$. Using~\eqref{Eq: Expsum_as_product}, we can immediately deduce the following lemma from~\eqref{Eq: primePowerEstimates} and~\eqref{Eq: HuaEstimate}, which is the analogue of~\cite[Lemma 5.1.]{DiagonalCubicHB}.
\begin{lemma}\label{Le: ExpSumEstimate}
It holds that 
\[
S_{\varpi^2}(\bm{c})\ll |\varpi|^{2+n}.
\]
In addition, if $(\varpi^k,\bm{c})=H_\varpi$ and there at least $m$ indices $i$ such that $(\varpi^k,c_i)=H_\varpi$, then 
\[
S_{\varpi^k}(\bm{c})\ll|\varpi|^{k+2(n-m)k/3+mk/2}|H_\varpi|^{m/4}.
\]
\end{lemma}
\subsection{Averages over square-full moduli}\label{Se: SectionSquarefullModuli}
Suppose we are given a set of $t$ indices $\mathcal{T}\subset\{1,\dots, n\}$ and positive integers $C_i$ for $i\in\mathcal{T}$. For $\bm{C}\coloneqq (C_i)_{i\in\mathcal{T}}$ we define 
 $\mathcal{R}(\bm{C})\subset\O^n$ to be the set of tuples $\bm{c}=(c_1,\dots, c_n)$ such that $|c_i|=\widehat{C}_i$ if $i\in \mathcal{T}$ and $c_j=0$ whenever $j\not\in\mathcal{T}$. Given $Y\in\z_{>0}$, we are interested in averages of the form 
 \begin{equation}
     \mathcal{A}(\mathcal{R}(\bm{C}), \widehat{Y})\coloneqq \sum_{\substack{\bm{c}\in\mathcal{R}(\bm{C})\\ F^*(\bm{c})\neq 0}}\sum_{\substack{r\in\O\\|r|= \widehat{Y}}}|S_{r}(\bm{c})|,
 \end{equation}
 where $r$ is restricted to square-full polynomials.
 \begin{lemma}\label{Le: ExpSumAverageSquarefull}
 With the notation from above, we have 
 \[
 \mathcal{A}(\mathcal{R}(\bm{C}),\widehat{Y})\ll_\varepsilon \widehat{Y}^{1+n/2+(n-t)/6}(\widehat{Y}\widehat{C})^{\varepsilon}\#\mathcal{R}(\bm{C}),
 \]
 where $\widehat{C}=\max_{i\in\mathcal{T}}\widehat{C}_i$.
 \end{lemma}
 The proof of Lemma~\ref{Le: ExpSumAverageSquarefull} is along the same lines as that of~\cite[Lemma 5.2]{DiagonalCubicHB}, and so we shall be brief. 
 \begin{proof}
 First of all, we introduce some notation. Fix $\bm{c}\in\mathcal{R}(\bm{C})$. For $r\in\O$ monic square-full, we write
  \begin{equation}\label{Eq: ProdDecomposition}
 r=r_*\prod_{i\in\mathcal{T}}r_i,
 \end{equation}
where the various coprime factors $r_*, r_i$ are defined as follows. We let $r_*$ be the product of those monic prime powers $\varpi^k$ such that $\varpi^k\parallel r$ and $k=2$ or $\varpi\nmid c_i$ for $i\in\mathcal{T}$. Moreover, for $i\in\mathcal{T}$, we define $r_i$ to be the product of monic prime powers $\varpi^k\parallel r$ such that $\varpi\mid c_i$, but $\varpi\nmid c_j$ for any $j\in\mathcal{T}$ with $j<i$. In particular, any $r_i$ is cube-full. Since all the factors in~\eqref{Eq: ProdDecomposition} are coprime, it follows from~\eqref{Eq: S_r(c) multiplicative} that 
 \[
 S_r(\bm{c})=S_{r_*}(\bm{c})\prod_{i\in\mathcal{T}}S_{r_i}(\bm{c}).
 \]
 Using the fact that $S_{\varpi^k}(\bm{c})=0$ if $\varpi\nmid F^*(\bm{c})$ for $k\geq 2$ and the estimates~\eqref{Eq: primePowerEstimates} and~\eqref{Eq: HuaEstimate}, we deduce that 
 \[
 S_r(\bm{c})\ll \eta(r,\bm{c})|r|^{1+n/2+(n-t)/6+\varepsilon}\prod_{i,j\in\mathcal{T}}\lvert\{r_i,c_j\}\rvert^{1/4},
 \]
 where $\eta(r,\bm{c})=1$ if $\varpi\mid F^*(\bm{c})$ for all primes $\varpi\mid r_*$ and $\eta(r,\bm{c})=0$ else. Let us now fix the absolute values of $r_*$ and of the various $r_i$'s, say $|r_*|=\widehat{Y}_*$ and $|r_i|=\widehat{Y}_i$, and denote their contribution to $\mathcal{A}(\mathcal{R}(\bm{C}),\widehat{Y})$ by $\mathcal{A}(Y_*, \bm{Y})$, where $\bm{Y}=(Y_i)_{i\in\mathcal{T}}$. We then have
 \[
 \mathcal{A}(Y_*,\bm{Y})\ll \widehat{Y}^{1+n/2+(n-t)/6+\varepsilon}\sum_{\substack{\bm{c}\in\mathcal{R}(\bm{C})\\F^*({\bm{c})\neq 0}}}\sum_{\substack{|r_i|=\widehat{Y}_i\\ i\in\mathcal{T}}}\prod_{i,j\in\mathcal{T}}\lvert\{r_i,c_j\}\rvert^{1/4}S_{\bm{c}},
 \]
 where we have suppressed the dependence of $r_*$ and of the $r_i$'s on $\bm{c}$ in the notation and where
 \[
 S_{\bm{c}} = \sum_{|r_*|=\widehat{Y}_*}\eta({r,\bm{c}}).
 \]
Heath-Brown's argument for estimating $S_{\bm{c}}$ goes through almost verbatim in our setting and gives $S_{\bm{c}}\ll (\widehat{Y}\widehat{C})^\varepsilon$. Therefore, we have
\[
\mathcal{A}(Y_*,\bm{Y})\ll \widehat{Y}^{1+n/2+(n-t)/6+\varepsilon}(\widehat{Y}\widehat{C})^{\varepsilon}\sum_{\substack{\bm{c}\in\mathcal{R}(\bm{C})\\F^*({\bm{c})\neq 0}}}\sum_{\substack{|r_i|=\widehat{Y}_i\\ i\in\mathcal{T}}}\prod_{i,j\in\mathcal{T}}\lvert\{r_i,c_j\}\rvert^{1/4}.
\]
To achieve the desired upper bound, we shall now only require that each $r_i$ is cube-full and that $\varpi\mid c_i$ whenever $\varpi\mid r_i$, so that in particular the $r_i$'s do not depend on $\bm{c}$ anymore. Thus, after setting 
\[
S(j)=\sum_{|c_j|=\widehat{C}_j}\prod_{i\in\mathcal{T}}|\{r_i,c_j\}|^{1/4},
\]
we obtain
\begin{equation}
    \mathcal{A}(Y_*,\bm{Y})\ll \widehat{Y}^{1+n/2+(n-t)/6+\varepsilon}(\widehat{Y}\widehat{C})^{\varepsilon}\sum_{\substack{|r_i|=\widehat{C}_i\\i\in\mathcal{T}}}\prod_{j\in\mathcal{T}}S(j).
\end{equation}
It is again straightforward to verify that Heath-Brown's argument continues to hold in our setting, yielding
\[
\sum_{\substack{|r_i|=\widehat{C}_i\\i\in\mathcal{T}}}\prod_{j\in\mathcal{T}}S(j)\ll \widehat{Y}^{(n+1)\varepsilon}\#\mathcal{R}(\bm{C}).
\]
With a new choice of $\varepsilon$, we conclude
\[
    \mathcal{A}(Y_*,\bm{Y})\ll\widehat{Y}^{1+n/2+(n-t)/6}(\widehat{Y}\widehat{C})^{\varepsilon}\#\mathcal{R}(\bm{C}),
\]
so that the statement of the lemma follows from the fact that there are only $\widehat{Y}^\varepsilon$ possibilities for admissible tuples $(Y_*,\bm{Y})$. 
 \end{proof}
 \section{Rational points on the dual hypersurface}\label{Se: DualForm}
In this section we study roots of the dual form $F^*$ of $F$ that was defined in~\eqref{Def: DualForm}. Our first goal is to find an upper bound for the number of solutions $F^*(\bm{c}) = 0$ with $\lvert \bm{c} \rvert \leq \widehat{C}$ when $\cha(K)>3$. In order to achieve this we closely follow the strategy of Heath-Brown~\cite[Section 7]{heath1983cubic}. The result of Lemma~\ref{lem.lin_ind_square_roots} is standard over the rational numbers, however we could not find a proof in the literature for our setting and so we included a proof here.

If $n = 4$ and $\cha(K)>3$ we call a solution $\bm{c}$ to $F^*(\bm{c}) = 0$ \emph{special} if $c_1, \hdots, c_4 \neq 0$ and  there are indices $i,j,k,l$ such that $\{i,j,k,l\} = \{1,2,3,4\}$ and
\[
(F_i^{-1} c_i^3)^{1/2} + (F_j^{-1} c_j^3)^{1/2} = (F_k^{-1} c_k^3)^{1/2} + (F_l^{-1} c_l^3)^{1/2} = 0
\]
holds for a suitable choice of square roots. We call a solution $\bm{c}$ to $F^*(\bm{c}) = 0$ \emph{ordinary} if it is not special. In particular, if $\cha(K)=2$ every solution is ordinary.
\begin{lemma} \label{lem.number_of_solutions_dual_form}
Assume $\cha(K)>3$. If $n = 6$, then the number of solutions to $F^*(\bm{c}) = 0$ with $\lvert \bm{c} \rvert \leq \widehat{C}$ is bounded by $O(\widehat{C}^{3+\varepsilon})$. Moreover, if $n = 4$, then the number of ordinary solutions to $F^*(\bm{c}) = 0$ with $\lvert \bm{c} \rvert \leq \widehat{C}$ is bounded by $O(\widehat{C}^{1+\varepsilon})$.
\end{lemma}
Before we can begin with the proof of this lemma, we need an auxiliary result. In the following we fix $\zeta \in \mathbb{F}_q^\times$ to be a representative of a non-trivial element in $\mathbb{F}_q^\times/\mathbb{F}_q^{\times,2}$. If $\mathrm{char}(\mathbb{F}_q) >2$ this certainly exists --- we may for example pick $\zeta$ to be a primitive root of $\mathbb{F}_q^\times$.
\begin{lemma} \label{lem.lin_ind_square_roots}
Suppose $\cha(K)>3$. Let $m_1, \hdots, m_n \in \mathcal{O}$ be a collection of distinct square-free polynomials such that each $m_i$ is either monic or has leading coefficient $\zeta$. Then $\{ \sqrt{m_1}, \hdots, \sqrt{m_n} \}$ is a linearly independent set over $K$.
\end{lemma}
\begin{proof}
We will prove the result by induction on $n$. The cases $1 \leq n \leq 3$ can easily be verified directly, so suppose $n \geq 4$. Assume for a contradiction that $\lambda_1, \hdots, \lambda_n \in K$ not all zero are such that
\[
\sum_{k=1}^n \lambda_k \sqrt{m_k} = 0.
\]
Note that we  may assume $\lambda_i \neq 0$ for all $i = 1, \hdots, n$ since otherwise the result would follow immediately from the induction hypothesis. In particular it is sufficient to show that there exists some index $k$ with $\lambda_k = 0$.
Since $n \geq 3$ there exist two distinct indices $i,j$ such that $m_i/m_j \notin \mathbb{F}_q^\times$. From the $n = 3$ case it follows that $K_{i,j} \coloneqq K(\sqrt{m_i}, \sqrt{m_j})$ is a Galois extension of degree $4$ over $K$. Thus there exists $\sigma \in \Gal(K_{i,j}/K)$ such that $\sigma(\sqrt{m_i}) = -\sqrt{m_i}$ and $\sigma(\sqrt{m_j}) = \sqrt{m_j}$. We may lift this to an element $\tilde{\sigma} \in \Gal (K^s/K)$ where $K^s$ is the separable closure of $K$. Then we find
\[
0 = \tilde{\sigma} \left( \sum_{k=1}^n \lambda_k \sqrt{m_k}\right) + \sum_{k=1}^n \lambda_k \sqrt{m_k} = 2\lambda_j \sqrt{m_j} + \sum_{k \neq i,j} \widetilde{\lambda}_k \sqrt{m_k},
\]
where $\widetilde{\lambda}_k \in \{0,2 \lambda_k \}$. From the induction hypothesis we get $\lambda_j = 0$, which yields the desired result as remarked above.
\end{proof}

\begin{proof}[Proof of Lemma~\ref{lem.number_of_solutions_dual_form}]
First note that $F^*(\bm{c}) = 0$ if and only if
\begin{equation} \label{eq.dual_form_vanishing_condition}
(F_1^{-1}c_1^3)^{1/2} + \cdots + (F_n^{-1} c_n^3)^{1/2} = 0,
\end{equation}
for a suitable choice of square roots. Let $m_k \in \mathcal{O}$ be a square-free polynomial, which is either monic or has leading coefficient $\zeta$. Say $i \in \mathcal{I}(k)$ if there exists some $d_i \in \mathcal{O}$ such that $F_i c_i^3 = m_kd_i^2$. From Lemma~\ref{lem.lin_ind_square_roots} we find that~\eqref{eq.dual_form_vanishing_condition} implies
\[
\sum_{i \in \mathcal{I}(k)}F_i^{-1}d_i = 0.
\]
We have $c_i^2 \mid m_k d_i^2$ and consequently $c_i \mid d_i$ since $m_k$ is square-free. Thus there exists $e_i \in \mathcal{O}$ such that $d_i = c_i e_i$. Substituting this into the relation $F_i c_i^3 = m_kd_i^2$ we find $c_i = m_k F_i^{-1} e_i^2$ and hence $d_i = c_ie_i = m_k F_i^{-1} e_i^3$. Therefore $F_i^{-1}d_i = m_k F_i \left( \frac{e_i}{F_i} \right)^3$ and the preceding display gives
\begin{equation} \label{eq.another_cubic_e_i}
\sum_{i \in \mathcal{I}(k)}F_i \left( \frac{e_i}{F_i} \right)^3 = 0.
\end{equation}
We will now estimate the number of solutions $\bm{e}$ to~\eqref{eq.another_cubic_e_i} such that $\lvert \bm{e} \rvert \leq \widehat{E} = \sqrt{\widehat{C}/\lvert m_k \rvert}$. 
This will then enable us to estimate the number of solutions of~\eqref{eq.dual_form_vanishing_condition}. Via Hölder's inequality and Hua's Lemma in this context (cf.~\cite[Lemma 5.12]{Browning2021}) we find
\[
\# \left\{\lvert \bm{e} \rvert \leq \widehat{E} \colon \sum_{i \in \mathcal{I}(k)}F_i \left( \frac{e_i}{F_i} \right)^3 = 0 \right\} \ll \begin{cases}
1 \quad &\text{if $\# \mathcal{I}(k) = 1$,} \\
\widehat{E}^{2+\varepsilon} &\text{if $2 \leq \# \mathcal{I}(k) \leq 4$,} \\
\widehat{E}^{\# \mathcal{I}(k)-2+\varepsilon} &\text{if $ 5 \leq \# \mathcal{I}(k) \leq 6$.}
\end{cases}
\]
Note that at this point it is crucial to assume $\cha(K)>3$, because the Weyl differencing argument in the proof of Hua's lemma breaks down otherwise. Therefore for a fixed partition $\bigsqcup_j \mathcal{I}(k_j) = \{1, \hdots, n\}$ corresponding to $\{m_{k_j} \}$  the number of $\lvert \bm{c} \rvert \leq \widehat{C}$ satisfying~\eqref{eq.dual_form_vanishing_condition} is bounded above by
\[
\prod_{j} \left( \frac{\widehat{C}}{\lvert m_{k_j} \rvert} \right)^{e_{k_j}/2+\varepsilon},
\]
where
\[
e_{k_j} = \begin{cases}
0, \quad &\text{if $\# \mathcal{I}(k_j) = 1$} \\
2, &\text{if $2 \leq \# \mathcal{I}(k_j) \leq 4$} \\
3, &\text{if $\# \mathcal{I}(k_j) = 5$} \\
4, &\text{if $\# \mathcal{I}(k_j) = 6$.}
\end{cases}
\]
By considering all possible square-free elements $\lvert m_{k_j} \rvert \ll \widehat{C}$, we see that the total number of solutions of~\eqref{eq.dual_form_vanishing_condition} corresponding to a fixed partition is bounded above by
\[
\sum_{\lvert m_{k_j} \rvert \leq \widehat{C}} \prod_{j} \left( \frac{\widehat{C}}{\lvert m_{k_j} \rvert} \right)^{e_{k_j}/2+\varepsilon} \ll \prod_j \widehat{C}^{e_{k_j}/2+\varepsilon}. 
\]
It is easily checked that for any possible partition this is bounded above by $O(\widehat{C}^{3+\varepsilon})$ if $n = 6$. Therefore the total number of solutions to $F^*(\bm{c}) = 0$ with $\lvert \bm{c} \rvert \leq \widehat{C}$ has the same upper bound. In the case $n=4$ one can similarly  obtain $O(\widehat{C}^{1+\varepsilon})$ for the number of solutions corresponding to any partition, except in the case where $\# \mathcal{I}(k_1) = \# \mathcal{I}(k_2) = 2$. But solutions arising from such partitions are precisely the special solutions. This finishes the proof of the lemma.
\end{proof}

 \section{Circle method}\label{Se:ActivationDelta}
As explained in the introduction, we are considering a diagonal cubic form $F\in\O[x_1,\dots,x_n]$ of the shape 
\[
F(\bm{x})=\sum_{i=1}^n F_ix_i^3,\quad F_i\in \O\setminus\{0\}.
\]
Recall from~\eqref{Eq: DeltaMethod} that the associated counting function can be written as 
\[
  N(w,P)=|P|^n\sum_{\substack{r\text{ monic}\\ |r|\leq \widehat{Q}}}|r|^{-n}\int_{|\theta|<|r|^{-1}\widehat{Q}^{-1}}\sum_{\bm{c}\in \O^n}S_r(\bm{c})I_r(\theta,\bm{c})\dd\theta.
\]
Throughout the parameter $Q$ is chosen in such a way that
\begin{equation}\label{Eq: SizeQ}
\lvert P\rvert^{3/2}\leq \widehat{Q} \leq q\lvert P\rvert^{3/2}
\end{equation}
ensuring that the measure of the set $\{|\theta|<|r|^{-1}\widehat{Q}^{-1}\}$ is $O(|P|^{-3})$ when $|r|=\widehat{Q}$. It follows from Lemma~\ref{lem.J_f_vanishing} that $I_r(\theta,\bm{c})$ vanishes unless $|\bm{c}| < |r||P|^{-1}\max\{q, H_F|P|^3\theta\}$. Since $H_F|P|^3 \lvert \theta \rvert \leq H_F |P|^3\widehat{Q}^{-1}|r|^{-1}$ and $|P|^3\widehat{Q}^{-1}|r|^{-1}\gg 1$, we can truncate the sum over $\bm{c}$ in~\eqref{Eq: DeltaMethod} at $|\bm{c}|\ll \widehat{C}$, where $\widehat{C}\coloneqq |P|^2\widehat{Q}^{-1}$.

We now split up $N(w,P)$ according to the quality of our available estimates into 
\[
N(w,P)=N_0(P) +E_1(P)+E_2(P),
\]
where 
\begin{align}
     N_0(P)&= |P|^n\sum_{\substack{r\text{ monic}\\ |r|\leq \widehat{Q}}}|r|^{-n}\int_{|\theta|<|r|^{-1}\widehat{Q}^{-1}}S_r(\bm{0})I_r(\theta,\bm{0})\dd\theta,\label{Def: N_0(P)}\\   E_1(P)&= |P|^n\sum_{\substack{r\text{ monic}\\ |r|\leq \widehat{Q}}}|r|^{-n}\int_{|\theta|<|r|^{-1}\widehat{Q}^{-1}}\sum_{\substack{\bm{c}\in\O^n\\F^*(\bm{c})\neq 0}}S_r(\bm{c})I_r(\theta,\bm{c})\dd\theta,\label{Def: E_1(P)}\\
    E_2(P)&= |P|^n\sum_{\substack{r\text{ monic}\\ |r|\leq \widehat{Q}}}|r|^{-n}\int_{|\theta|<|r|^{-1}\widehat{Q}^{-1}}\sum_{\substack{\bm{c}\in\O^n\setminus\{\bm{0}\}\\F^*(\bm{c})= 0}}S_r(\bm{c})I_r(\theta,\bm{c})\dd\theta.\label{Def: E_2(P)}
\end{align}
For $n=4$ we will later divide the term $E_2(P)$ into special and ordinary solutions of $F^*(\bm{c})=0$ as defined in Section~\ref{Se: DualForm}. Usually one expects that the main term in an asymptotic formula for $N(w,P)$ should come from $N_0(P)$. As we are only interested in an upper bound for $N(w,P)$, the contribution from $N_0(P)$ will be rather straightforward to deal with. Handling the terms $E_1(P)$, $E_2(P)$ turns out to be a more challenging task and will occupy most of the remainder of our work. For $E_1(P)$ we can make use of the full power of our exponential sum estimates, in particular we gain an extra saving when averaging $S_r(\bm{c})$ over $r$. This is not possible for $E_2(P)$, but we shall benefit from the sparsity of $\bm{c}$'s such that $F^*(\bm{c})=0$, at least for ordinary solutions when $n=4$.

\subsection{Contribution from $N_0(P)$}
For this we write again $r=r_1r_2$, where $r_1$ is cube-free and $r_2$ is cube-full. It thus follows from~\eqref{Eq: DeligneS_r(c)} and Lemma~\ref{Le: ExpSumEstimate} with $m=0$ that 
\[
S_r(\bm{c})\ll |r_1|^{1+n/2+\varepsilon}|r_2|^{1+2n/3+\varepsilon}.
\]
From Lemma~\ref{lem.I(0) estimate} we obtain the estimate $I_r(\bm{0})\ll \lvert P \rvert^{-3 + \varepsilon}$. We thus get
\begin{align*}
    N_0(P) &\ll |P|^{n-3+\varepsilon}\sum_{|r_1|\leq \widehat{Q}}|r_1|^{-n}S_{r_1}(\bm{c})\sum_{|r_2|\leq \widehat{Q}/|r_1|}|r_2|^{-n}S_{r_2}(\bm{c})\\
    &\ll |P|^{n-3+\varepsilon}\sum_{|r_1|\leq \widehat{Q}}|r_1|^{1-n/2}\sum_{|r_2|\leq \widehat{Q}/|r_1|}|r_2|^{1-n/3}\\
    &\ll |P|^{n-3+\varepsilon},
\end{align*}
since there are $O(\widehat{Y}^{1/3})$ cube-full $r_2$ with $|r_2|=\widehat{Y}$.

\subsection{Contribution from $E_1(P)$}
We begin with some preparations for the term $E_1(P)$.  Let $0\leq Y \leq Q$ and fix the absolute value of $r$ to be $\widehat{Y}$.  As in Section~\ref{Se: SectionSquarefullModuli}, we will also fix a set of indices $\mathcal{T}\subset \{1,\dots, n\}$ of cardinality $t$, as well as a tuple $\bm{C}=(C_i)_{i\in\mathcal{T}}$, where $1\leq C_i\leq C$ and denote by $\mathcal{R}(\bm{C})$ the set of vectors $\bm{c}=(c_1,\dots,c_n)\in\O^n$ such that $|c_i|=\widehat{C}_i$ if $i\in\mathcal{T}$ and $c_j=0$ if $j\not\in\mathcal{T}$. Let us put $\mathcal{C}=\max_{i\in\mathcal{T}}C_i$, so that $|\bm{c}|=\widehat{\mathcal{C}}$ whenever $\bm{c}\in\mathcal{R}(\bm{C})$. We then define $E_1(\mathcal{R}(\bm{C}),\widehat{Y})$ to be the contribution coming from $\bm{c}\in \mathcal{R}(\bm{C})$ and $|r|=\widehat{Y}$ in the definition of $E_1(P)$ given in~\eqref{Def: E_1(P)}. Explicitly, this means
\begin{equation}
    E_1(\mathcal{R}(\bm{C}),\widehat{Y}) = \frac{|P|^n}{\widehat{Y}^n}\sum_{\substack{\bm{c}\in\mathcal{R}(\bm{C})\\ F^*(\bm{c})\neq 0}}\sum_{\substack{r\text{ monic}\\|r|=\widehat{Y}}}S_r(\bm{c})I_{\widehat{Y}}(\bm{c}),
\end{equation}
where
\[
I_{\widehat{Y}}(\bm{c})=\int_{|\theta|<\widehat{Y}^{-1}\widehat{Q}^{-1}}I_r(\theta,\bm{c})\dd\theta.
\]
The definition of $I_{\widehat{Y}}(\bm{c})$ makes sense by Lemma~\ref{Le: I_r(theta,c) deps on |r|}, which shows that the value of the double integral in the definition of $I_{\widehat{Y}}(\bm{c})$ only depends on the absolute value of $r$ for $\bm{c}$ fixed.\\

Note that there are $Q+1\ll |P|^\varepsilon$ possibilities for $Y$ and $O(C^n)=O(|P|^\varepsilon)$ choices for $\bm{C}$. In particular, if we can show that $E_1(\mathcal{R}(\bm{C}),\widehat{Y})\ll |P|^{3n/4-3/2+\varepsilon}$ holds, then the same estimate will be true for $E_1(P)$ with a new value of $\varepsilon >0$. Next we tansform $E_1(P)$ in such a way that Lemma~\ref{Le: Hasse-WeilExpBound} and Lemma~\ref{Le: ExpSumAverageSquarefull} are applicable. For this we write $r=b_1'b_1r_2$, where $r_2$ is the square-full part of $r$ and $b_1'b_1$ is the square-free part of $r$. Moreover, if we let $S$ be the set of prime divisors of $\Delta_FF^*(\bm{c})$, then we further require that $(b_1,S)=1$ and each prime $\varpi\mid b_1'$ satisfies $\varpi\in S$. It then follows from~\eqref{Eq: S_r(c) multiplicative} that 
\begin{equation}\label{Eq: E_1(R(C),Y)}
    E_1(\mathcal{R}(\bm{C}),\widehat{Y}) = \frac{|P|^n}{\widehat{Y}^{(n-1)/2}}\sum_{\substack{\bm{c}\in\mathcal{R}(\bm{C})\\ F^*(\bm{c})\neq 0}} I_{\widehat{Y}}(\bm{c})\sum_{\substack{|r_2|\leq \widehat{Y}}}\frac{S_{r_2}(\bm{c})}{|r_2|^{(n+1)/2}}\sum_{\substack{|b_1'|\leq \frac{\widehat{Y}}{|r_2|}}}\frac{S_{b_1'}(\bm{c})}{|b_1'|^{(n+1)/2}}\sum_{\substack{|b_1|=\frac{ \widehat{Y}}{|r_2b_1'|}\\ (b_1,S)=1}}\frac{S_{b_1}(\bm{c})}{|b_1|^{(n+1)/2}}.
\end{equation}
We can now apply Lemma~\ref{Le: Hasse-WeilExpBound} to the innermost sum to obtain
\begin{equation}\label{Eq: b_1conribution}
    \sum_{\substack{|b_1|=\frac{ \widehat{Y}}{|r_2b_1'|}\\ (b_1,S)=1}}\frac{S_{b_1}(\bm{c})}{|b_1|^{(n+1)/2}}\ll \widehat{C}^{\varepsilon}(\widehat{Y}|r_2b_1'|^{-1})^{1/2+\varepsilon}.
\end{equation}
Moreover, by~\eqref{Eq: DeligneS_r(c)} and~\eqref{Eq: S_r(c) multiplicative} we also have 
\begin{equation}
    \sum_{\substack{|b_1'|\leq \frac{\widehat{Y}}{|r_2|}}}\frac{|S_{b_1'}(\bm{c})|}{|b_1'|^{n/2+1}}\ll |P|^{\varepsilon} \sum_{|b_1'|\leq \widehat{Y}/|r_2|}\frac{|(b_1', \nabla F^*(\bm{c}))|^{1/2}}{|b_1'|^{1/2}} \ll |P|^\varepsilon,
\label{Eq: b_1'contribution}
\end{equation}
where we used that there at most $O((\widehat{Y}|r_2|^{-1}|F^*(\bm{c})|)^\varepsilon)=O(|P|^{\varepsilon})$ possibilities for square-free $b_1'$ whose prime divisors are restricted to $S$ with $|b_1'|\leq \widehat{Y}|r_2|^{-1}$. After inserting~\eqref{Eq: b_1conribution} and~\eqref{Eq: b_1'contribution} into~\eqref{Eq: E_1(R(C),Y)}, we see that
\begin{align*}
    E_1(\mathcal{R}(\bm{C}),\widehat{Y})&\ll \frac{|P|^{n+\varepsilon}}{\widehat{Y}^{n/2-1}}\sum_{\substack{\bm{c}\in\mathcal{R}(\bm{C})\\ F^*(\bm{c})\neq 0}} |I_{\widehat{Y}}(\bm{c})|\sum_{\substack{|r_2|\leq \widehat{Y}}}\frac{|S_{r_2}(\bm{c})|}{|r_2|^{n/2+1}}.
\end{align*}
We can now estimate $I_{\widehat{Y}}(\bm{c})$ with Lemma~\ref{lem.integral_estimate_for_I_r}:
\begin{align*}
    I_{\widehat{Y}}(\bm{c})&\ll \widehat{Y}^{-1}\widehat{Q}^{-1}\prod_{i = 1}^n \min \left\{ \left( \frac{\lvert P \rvert \lvert \bm{c}\rvert}{\widehat{Y}} \right)^{-1/4}, \left( \frac{\lvert P \rvert \lvert c_i \rvert}{\widehat{Y}} \right)^{-1/2} \right\}\\
    &= \widehat{Y}^{-1}\widehat{Q}^{-1}\left(\frac{\widehat{Y}}{|P|\widehat{\mathcal{C}}}\right)^{(n-t)/4}\prod_{i\in\mathcal{T}}\min \left\{ \left( \frac{\lvert P \rvert  \widehat{\mathcal{C}}}{\widehat{Y}} \right)^{-1/4}, \left( \frac{\lvert P \rvert \widehat{C}_i}{\widehat{Y}} \right)^{-1/2} \right\},
\end{align*}
where we used that $\min\left\{ \left( \frac{\lvert P \rvert  \widehat{\mathcal{C}}}{\widehat{Y}} \right)^{-1/4}, \left( \frac{\lvert P \rvert \lvert c_i \rvert}{\widehat{Y}} \right)^{-1/2} \right\}=(|P|\widehat{\mathcal{C}} \, \widehat{Y}^{-1})^{-1/4}$ if $i\not\in\mathcal{T}$. Denote the last product above by $\Pi$. Then after dividing $r_2$ into $q$-adic ranges, Lemma~\ref{Le: ExpSumAverageSquarefull} implies
\begin{align*}
E_1(\mathcal{R}(\bm{C}),\widehat{Y})&\ll \frac{|P|^{n+\varepsilon}}{\widehat{Y}^{n/2}\widehat{Q}}\left(\frac{\widehat{Y}}{|P|\widehat{\mathcal{C}}}\right)^{(n-t)/4}\Pi\sum_{\substack{\bm{c}\in\mathcal{R}(\bm{C})\\ F^*(\bm{c})\neq 0}}\sum_{\substack{|r_2|\leq \widehat{Y}}}\frac{|S_{r_2}(\bm{c})|}{|r_2|^{n/2+1}}\\
&\ll \frac{|P|^{n+\varepsilon}}{\widehat{Y}^{n/2}\widehat{Q}}\left(\frac{\widehat{Y}}{|P|\widehat{\mathcal{C}}}\right)^{(n-t)/4}\widehat{Y}^{(n-t)/6}\Pi \#\mathcal{R}(\bm{C}).
\end{align*}
From the fact that $\#\mathcal{R}(\bm{C})\ll \prod_{i\in\mathcal{T}}\widehat{C}_i$ we deduce that
\begin{align*}
    \#\mathcal{R}(\bm{C})\Pi&\ll \prod_{i\in\mathcal{T}}\min \left\{ \widehat{C}_i\left(\frac{\widehat{Y}}{|P|\widehat{\mathcal
    C}}\right)^{1/4},\left(\frac{\widehat{C}_i\widehat{Y}}{|P|}\right)^{1/2}\right\}\\
    &\ll \widehat{\mathcal{C}} \,^t\left(\frac{\widehat{Y}}{|P|\widehat{\mathcal
    C}}\right)^{t/4}\min\left\{1,\frac{\widehat{Y}}{|P|\widehat{\mathcal{C}}}\right\}^{t/4},
\end{align*}
where we used that $\widehat{C}_i\leq \widehat{\mathcal{C}}$. Recalling~\eqref{Eq: SizeQ}, we therefore have
\[
E_1(\mathcal{R}(\bm{C}),\widehat{Y})\ll \frac{|P|^{n-3/2+\varepsilon}}{\widehat{Y}^{n/2}}\left(\frac{\widehat{Y}}{|P|\widehat{\mathcal{C}}}\right)^{n/4}\widehat{Y}^{(n-t)/6}\widehat{\mathcal{C}}\,^t\min\left\{1,\frac{\widehat{Y}}{|P|\widehat{\mathcal{C}}}\right\}^{t/4}.
\]
One easily sees that the expression above is maximal either at $t=0$ or $t=n$. For $t=0$, we get
\begin{align*}
    \frac{|P|^{n-3/2+\varepsilon}}{\widehat{Y}^{n/2}}\left(\frac{\widehat{Y}}{|P|\widehat{\mathcal{C}}}\right)^{n/4}\widehat{Y}^{n/6}& = |P|^{3n/4-3/2+\varepsilon} \widehat{Y}^{-n/12}\widehat{\mathcal{C}}^{-n/4}\\
    &\ll |P|^{3n/4-3/2+\varepsilon}
\end{align*}
as desired. For $t=n$, we have 
\begin{align*}
    \frac{|P|^{n-3/2+\varepsilon}}{\widehat{Y}^{n/2}}\left(\frac{\widehat{Y}}{|P|\widehat{\mathcal{C}}}\right)^{n/4}\widehat{\mathcal{C}}^n\min\left\{1,\frac{\widehat{Y}}{\widehat{\mathcal{C}}|P|}\right\}^{n/4}&\ll |P|^{n/2-3/2+\varepsilon}\widehat{\mathcal{C}}^{n/2}\\
    &\ll |P|^{3n/4-3/2+\varepsilon}
\end{align*}
since $\widehat{\mathcal{C}}\leq \widehat{C}\ll |P|^{1/2}$. This finishes our treatment of $E_1(P)$.
\subsection{Contribution from $E_2(P)$ for ordinary solutions} \label{subsec:E_2(P)}
Now we turn our attention to the term $E_2(P)$. For $n=4$ we further divide it into $E_2(P)=E_2^{\text{ord}}(P)+E_2^\text{spec}(P)$, where $E_2^\text{spec}(P)$ is restricted to special solutions of $F^*(\bm{c})=0$ in the sense of Section~\ref{Se: DualForm} and $E_2^\text{ord}(P)$ to ordinary solutions of $F^*(\bm{c})=0$. In this section we deal with $E_2(P)$ for $n=6$ and $E_2^\text{ord}(P)$ for $n=4$. \\

We shall again fix the absolute value of $r$ to be $\widehat{Y}$ for some $0\leq Y\leq Q$ and the absolute value of $\bm{c}$ to be $\widehat{\mathcal{C}}$ for some $0< \mathcal{C}\leq C$. We will then consider the sum
\[
E_2(Y,\mathcal{C})\coloneqq \frac{|P|^n}{\widehat{Y}^n}\sum_{\substack{|\bm{c}|=\widehat{\mathcal{C}}\\ F^*(\bm{c})=0}}\sum_{\substack{r\text{ monic}\\ |r|=\widehat{Y}}}S_r(\bm{c})I_{\widehat{Y}}(\bm{c}),
\]
where the sum over $\bm{c}$ is restricted to ordinary solutions of $F^*(\bm{c})=0$ for $n=4$. Once we have shown $E_2(Y,\mathcal{C})\ll |P|^{3n/4-3/2+\varepsilon}$ the same estimate will follow for $E_2(P)$ for $n=6$ and for $E_2^\text{ord}(P)$ for $n=4$, because there are only $O(|P|^\varepsilon)$ possible pairs of $Y$'s and $\mathcal{C}$'s.\\

\begin{lemma} \label{lem.few dual form solutions gives good E_2}
    Let $F$ be a non-singular cubic form in $4$ or $6$ variables, and let $F^*$ be its dual form. Suppose there exists some $\eta >0$ such that for any $\widehat{\mathcal{C}} \geq 1$  the following bound holds
    \[
    \# \{ \bm{x} \in \mathcal{O}^n \colon \text{$\bm{x}$ is an ordinary solution to } F^*(\bm{x}) = 0 , \lvert \bm{x} \rvert \leq \widehat{\mathcal{C}} \} \ll \widehat{\mathcal{C}}^{n-3+\eta}.
    \]
    Then we have
    \[
    E_2(P) \ll \lvert P \rvert^{3n/4-3/2 +\eta/2+ \varepsilon}.
    \]
\end{lemma}

\begin{proof}
If $D=\deg F^*$, then we see from~\eqref{Def: DualForm} and Lemma~\ref{Le: DualChar2} that $F^*$ has non-zero monomials of the form $G_ix_i^D$ for every $i=1,\dots, n$. In particular, if $|\bm{c}|=\widehat{\mathcal{C}}$ and $F^*(\bm{c})=0$, then there must be at least two indices $i\neq j$ such that $\widehat{\mathcal{C}}\ll |c_i|\ll|c_j|\ll \widehat{\mathcal{C}}$. Therefore, from Lemma~\ref{lem.integral_estimate_for_I_r} we deduce
\begin{equation}\label{Eq: IntEstimateE_2}
    I_{\widehat{Y}}(\bm{c})\ll \frac{\widehat{\mathcal{C}}}{|P|^2\widehat{Y}}\prod_{i=1}^n\min\left\{\left(\frac{\widehat{Y}}{|P||c_i|}\right)^{1/2},\left(\frac{\widehat{Y}}{|P|\widehat{\mathcal{C}}}\right)^{1/4}\right\}\ll \left(\frac{\widehat{Y}}{|P|\widehat{\mathcal{C}}}\right)^{(n-2)/4}|P|^{-3}.
\end{equation}
Next we deal with the sum $S_r(\bm{c})$. Write $r=r_1r_2r_3$ into coprime monic factors $r_i$, where $r_1$ is cube-free, $r_2$ is cube-full and each prime divisor of $r_3$ divides $\prod F_i$. \\

Let us begin with $S_{r_2}(\bm{c})$. Suppose $\varpi^k\parallel r_2$ and write $H_\varpi=(\varpi^k, \bm{c})$. It follows that $\bm{c}=H_\varpi \bm{c}'$ for some $\bm{c}'\in\O^n$ with $(\varpi,\bm{c}')=1$. It is again easy to see that any prime divisor of the coefficients $G_i$ of the top-degree monomials $x_i^D$ of $F^*$ divides $\prod F_i$. In particular, if $H_\varpi\neq \varpi^k$, then $F^*(\bm{c}')=0$ implies that at least two entries of $\bm{c}'$ are coprime to $\varpi$. On the other hand, if $H_\varpi=\varpi^k$, then $(\varpi^k,c_i)=\varpi^k$ for every $i=1,\dots, n$, so that in any case there are always least two distinct indices $i\neq j$ such that $(\varpi^k,c_i)=(\varpi^k,c_j)=H_\varpi$. Consequently it follows from Lemma~\ref{Le: ExpSumEstimate} with $m=2$ that 
\[
S_{r_2}(\bm{c})\ll |r_2|^{2/3+2n/3+\varepsilon}|H|^{1/2},
\]
where $H=\prod_{\varpi\mid r_2} H_\varpi$ divides each entry of $\bm{c}$.

In addition,~\eqref{Eq: DeligneS_r(c)} and Lemma~\ref{Le: ExpSumEstimate} give us $S_{r_1}(\bm{c})\ll |r_1|^{1+n/2+\varepsilon}$ and~\eqref{Eq: HuaEstimate} tells us that $S_{r_3}(\bm{c})\ll |r_3|^{1+2n/3+\varepsilon}$. To sum up, we have 
\begin{equation*}
S_r(\bm{c})\ll |r|^\varepsilon |r_1|^{1+n/2}|r_2|^{2/3+2n/3}|r_3|^{1+2n/3}|H|^{1/2}.
\end{equation*}
Let us fix $|r_i|=\widehat{Y}_i$, where $0\leq Y_i\leq Y$ and $Y_1+Y_2+Y_3=Y$. We want to give an upper bound for
\[
\mathcal{S}\coloneqq \sum_{|r_i|=\widehat{Y}_i, i=1,2,3}\sum_{\substack{|\bm{c}|=\widehat{C}\\ F^*(\bm{c})=0}}|S_r(\bm{c})|.
\]
Taking into account that the number of available $r_1$ and $r_3$ is $O(\widehat{Y}_1)$ and $O(|P|^\varepsilon)$ respectively, we see that 
\begin{align*}
    \mathcal{S} &\ll |P|^\varepsilon \widehat{Y}_1^{2+n/2}\widehat{Y}_2^{2/3+2n/3}\widehat{Y}_3^{1+2n/3}\sum_{|r_2|=\widehat{Y}_2}\sum_{H\mid r_2}|H|^{1/2}\sum_{\substack{|\bm{c}|=\widehat{C}/|H| \\ F^*(\bm{c})=0}}1\\
    &\ll |P|^\varepsilon \widehat{C}^{n-3+\eta} \widehat{Y}_1^{2+n/2}\widehat{Y}_2^{2/3+2n/3}\widehat{Y}_3^{1+2n/3} \sum_{|r_2|=\widehat{Y}_2}\sum_{H\mid r_2} |H|^{7/2-n-\eta},
\end{align*}
where we used the main assumption of the lemma in order to bound the number of ordinary solutions of $F^*(\bm{c})=0$ with $|\bm{c}|=\widehat{C}/|H|$ for the second inequality. Since $n \geq 4$ clearly $7/2-n - \eta \leq 0$ holds and since the number of available $r_2$ is $O(\widehat{Y}_2^{1/3})$, it follows that
\begin{equation}\label{Eq:EstimateS}
\mathcal{S}\ll |P|^\varepsilon\widehat{C}^{n-3+\eta}\widehat{Y}_1^{2+n/2}\widehat{Y}_2^{1+2n/3}\widehat{Y}_3^{1+2n/3}\ll |P|^{\varepsilon}\widehat{C}^{n-3+\eta}\widehat{Y}^{2+n/2},
\end{equation}
because $2+n/2\geq 1+2n/3$ for $n\leq 6$. As there are only $O(|P|^{\varepsilon})$ possibilities for permissible triples $(Y_1,Y_2,Y_3)$, we deduce from~\eqref{Eq: IntEstimateE_2} and~\eqref{Eq:EstimateS} that
\begin{align*}
    E_2(Y,\widehat{C})  &\ll |P|^{3n/4-5/2+\varepsilon}\widehat{Y}^{3/2-n/4}\widehat{\mathcal{C}}^{3n/4-5/2+\eta}.
\end{align*}
In particular, since $\widehat{\mathcal{C}}\ll |P|^{1/2}$ and $\widehat{Y}\ll |P|^{3/2}$, we thus obtain 
\begin{align*}
E_2(Y,\mathcal{C})&\ll |P|^{3n/4-5/2+\varepsilon}|P|^{9/4-3n/8}|P|^{3n/8-5/4+\eta/2}\\
&\ll |P|^{3n/4-3/2+\eta/2+ \varepsilon},
\end{align*}
which completes the proof.
\end{proof}
At this point our treatment of $E_2(P)$ differs depending on the characteristic of $K$.

If $\mathrm{char}(K) >3$, then by virtue of Lemma~\ref{lem.number_of_solutions_dual_form} we know that the number of ordinary solutions of the dual form $F^*(\bm{c}) = 0$ such that $\lvert \bm{c} \rvert \leq \widehat{\mathcal{C}}$ is bounded by $O(\widehat{\mathcal{C}}^{n-3+\varepsilon})$. Therefore Lemma~\ref{lem.few dual form solutions gives good E_2} implies
\[
E_2^{\mathrm{ord}}(P) \ll \lvert P \rvert^{3n/4-3/2+\varepsilon} \quad \text{and} \quad E_2(P) \ll \lvert P \rvert^{3n/4-3/2+\varepsilon},
\]
for $n=4$ and $n=6$, respectively. This finishes our treatment of $E_2(P)$ in this case.

If $\mathrm{char}(K) = 2$, then we need to argue differently. We begin by considering the case when $n=6$. According to Lemma~\ref{Le: DualChar2} the dual form takes the shape of a non-singular diagonal cubic form. In particular, we can trivially bound the number of solutions to $F^*(\bm{c}) = 0$ such that $\lvert \bm{c} \rvert \leq \widehat{\mathcal{C}}$ by $O(\widehat{\mathcal{C}}^6) = O(\widehat{\mathcal{C}}^{n-3+\eta})$, where $\eta = 3$. Therefore, Lemma~\ref{lem.few dual form solutions gives good E_2} gives
\[
E_2(P) \ll \lvert P \rvert^{3n/4-3/2+\eta/2+\varepsilon} = \lvert P \rvert^{n-3+\eta/2+\varepsilon}.
\]
This, together with our bounds for $N_0(P)$ and $E_1(P)$ established earlier in this section, shows that
\[
N(P) \ll \lvert P \rvert^{n-3+\eta/2 + \varepsilon}.
\]
This holds for any non-singular, diagonal cubic form over $K$ when $\mathrm{char}(K) = 2$. In particular, as a result we can bound the number of solutions to $F^*(\bm{c}) = 0$ with  $\lvert \bm{c} \rvert \leq \widehat{\mathcal{C}}$ by $O(\widehat{\mathcal{C}}^{n-3+\eta/2+\varepsilon})$. Another application of Lemma~\ref{lem.few dual form solutions gives good E_2} yields
\[
E_2(P) \ll \lvert P \rvert^{3n/4-3/2+\eta/4+\varepsilon}
\]
and we may argue as above to deduce
\[
N(P) \ll \lvert P \rvert^{n-3+\eta/4 + \varepsilon}.
\]
If we repeat this process $k$-times, where $2^{-k+1} \leq \varepsilon$ we find
\[
E_2(P) \ll \lvert P \rvert^{3n/4-3/2+2\varepsilon},
\]
which concludes our treatment for $E_2(P)$ in this case.

On the other hand, if $n=4$  we can trivially estimate the number of solutions to $F^*(\bm{c})=0$ of bounded height $\widehat{\mathcal{C}}$ by $O(\widehat{\mathcal{C}}^4) = O(\widehat{\mathcal{C}}^{n-3+\eta})$, where $\eta = 3$. Lemma~\ref{lem.few dual form solutions gives good E_2} then yields
\[
E_2(P) \ll \lvert P \rvert^{3n/4-3/2+\eta/2+\varepsilon} = \lvert P \rvert^{n-3+1/2+\eta/2+\varepsilon},
\]
which in turn implies
\[
N(P) \ll \lvert P \rvert^{n-3+1/2+\eta/2+\varepsilon}. 
\]
Repeating this process $k$-times, where $k > 1/\varepsilon$ we thus find
\[
E_2(P) \ll \lvert P \rvert^{3n/4-3/2+1/2+ 2\varepsilon} = \lvert P \rvert^{2+2\varepsilon}.
\]

\section{Waring's problem and weak approximation}\label{Se:WaringWA}
Having completed our task for $n=6$, we will now apply it to Waring's problem and weak approximation for diagonal cubic hypersurfaces of dimension at least $5$. 
\subsection{Waring's problem for $n\geq 7$}
Recall that $\mathbb{J}_{q}^3[t]$ is the additive closure of all cubes in $\O$. Given $P\in \mathbb{J}_{q}^3[t]$, we define $B\coloneqq \left\lceil \frac{\deg(P)}{3}\right\rceil +1$ and the counting function
\[
R_n(P)\coloneqq \#\{\bm{x}\in\O^n\colon |\bm{x}|<\widehat{B}, x_1^3+\cdots +x_n^3=P\}.
\]
Our next goal is to deduce Theorem~\ref{Th:Waring} from our findings. We shall accomplish this goal with a classical version of the circle method. For $\alpha\in\TT$, we define
\[
T(\alpha)\coloneqq \sum_{\substack{x\in\O\\ |x|<\widehat{B}}}\psi(\alpha x^3).
\]
It then follows from~\eqref{Eq: OrthogonalityOfCharacters} that we can write our counting function as
\[
R_n(P)=\int_\TT T(\alpha)^n\psi(-\alpha P)\dd\alpha.
\]
We then define our set of major arcs to be
\[
\mathfrak{M}\coloneqq \bigcup_{\substack{|r|\leq \widehat{B}\\ r\text{ monic}}}\bigcup_{\substack{|a|<|r|\\ (a,r)=1}}\{\alpha \in \TT\colon |r\alpha-a|<\widehat{B}^{-2}\}
\]
and $\mathfrak{m}\coloneqq \TT\setminus\mathfrak{M}$ constitutes our set of minor arcs. The following lemma is a consequence of~\cite[Theorem 30]{Kubota1974}.
\begin{lemma}\label{Thm: MajorArcsWaring}
Suppose $\cha(K)\nmid 3$ and $n\geq 7$. Then there exists $\delta>0$ such that for all $P\in\mathbb{J}_{q}^3[t]$ we have
\[
\int_{\mathfrak{M}}T(\alpha)^n\psi(-\alpha P)\dd\alpha=\mathfrak{S}(P)\sigma_\infty(P)\widehat{B}^{n-3}+O\left(\widehat{B}^{n-3-\delta}\right),
\]
where $\mathfrak{S}(P)$ and $\sigma_\infty(P)$ are the singular series and singular integral associated to $P$. Furthermore, they satisfy
\[
1\ll \mathfrak{S}(P)\sigma_\infty(P)\ll 1.
\]
\end{lemma}
\begin{remark}
In fact, Kubota states Lemma~\ref{Thm: MajorArcsWaring} only for $n\geq 10$. However, as explained by Liu--Wooley in~\cite[Lemma 5.2]{LiuWooley2010}, this is a result of an oversight and Kubota's argument already works for $n\geq 7$.
\end{remark}

We now have
\begin{equation}\label{Eq: WaringMinor}
    \left\lvert \int_{\mathfrak{m}}T(\alpha)^n\psi(-\alpha P)\dd\alpha \right\rvert \leq \sup_{\alpha\in\mathfrak{m}}|T(\alpha)|^{n-6}\int_{\TT}|T(\alpha)|^6\dd\alpha .
\end{equation}
If $\alpha\in \mathfrak{m}$, then~\eqref{Eq: DirichletDissection} with $\widehat{Q}=\widehat{B}$ implies the existence of $a,r \in \O$ with $r$ monic such that $|a|<|r|\leq \widehat{B}$, $(a,r)=1$ and $|r\alpha -a|<\widehat{B}^{-1}$. As $\alpha \in \mathfrak{m}$, we must have $|\alpha-a/r|\geq \widehat{B}^{-2}|r|^{-1}$. Under these circumstances Weyl's inequality, see~\cite[Lemma 5.10]{Browning2021} for $\cha(K)>3$ and~\cite[Proposition IV.4]{car1992sommes} for $\cha(K)=2$, guarantees the existence of $\delta >0$ such that
\begin{equation}\label{Eq:EstimateMinorWaringPtwise}
    \sup_{\alpha\in\mathfrak{m}}|T(\alpha)|^{n-6}\ll \widehat{B}^{(n-6)(1-\delta)}.
\end{equation}
Since
\[
\int_{\TT}|T(\alpha)|^6\dd\alpha = \#\{\bm{x}\in \O^6\colon |\bm{x}|<\widehat{B}, x_1^3+x_2^3+x_3^3=x_4^3+x_5^3+x_6^3\},
\]
Theorem~\ref{Th:TheTheorem} implies 
\begin{equation}\label{Eq:EstimateMinorWaringMeanValue}
    \int_{\TT}|T(\alpha)|^6\dd\alpha \ll \widehat{B}^{3+\varepsilon}.
\end{equation}
Plugging~\eqref{Eq:EstimateMinorWaringPtwise} and~\eqref{Eq:EstimateMinorWaringMeanValue} into~\eqref{Eq: WaringMinor} yields
\begin{align*}
    \int_{\mathfrak{m}}T(\alpha)^n \psi(-\alpha P) \dd\alpha &\ll \widehat{B}^{(n-6)(1-\delta)+3+\varepsilon}\\
    &= \widehat{B}^{n - 3 -\delta(n-6)+\varepsilon}.
\end{align*}
After choosing $\varepsilon= \delta(n-6)/2$, we see that the contribution of the minor arcs is 
\[
\int_{\mathfrak{m}}T(\alpha)^n\psi(-\alpha P)\dd\alpha\ll \widehat{B}^{n-3 -\delta(n-6)/2}.
\]
Since $n\geq 7$, combining this with Lemma~\ref{Thm: MajorArcsWaring} therefore completes the proof of Theorem~\ref{Th:Waring}.
\subsection{Weak approximation for cubic diagonal hypersurfaces}
We will show that weak approximation holds for the diagonal cubic hypersurface defined by $F(\bm{x}) = \sum_{i=1}^n F_i x_i^3$ if $n \geq 7$. 
Fix  $\bm{x}_0 \in \TT^n$, $M \in \mathcal{O}$, $\bm{b} \in \mathcal{O}^n$ and $N \in \mathbb{Z}_{\geq 0}$ such that $\lvert \bm{b} \rvert < \lvert M \rvert$ and such that $N$ is bounded in terms of $M$. Define the weight function $\widetilde{w} \colon K_\infty^n \rightarrow \mathbb{R}$ via
\[
\widetilde{w}(\bm{x}) = \begin{cases}
1 \quad &\text{if $\lvert \bm{x}-\bm{x}_0 \rvert < \widehat{N}^{-1}$}, \\
0 &\text{otherwise.}
\end{cases}
\]
Further for $P \in \mathcal{O}$ we introduce the counting function
\[
N(P, \widetilde{w}) \coloneqq \sum_{\substack{\bm{x} \in \mathcal{O}^n \\ F(M\bm{x} + \bm{b}) = 0}} \widetilde{w}\left(\frac{M\bm{x} + \bm{b}}{P}\right).
\]
As usual, we can write this as an integral over an exponential sum
\[
N(P, \widetilde{w})  = \int_{ \TT} \widetilde{S}(\alpha) \dd \alpha,
\]
where
\[
\widetilde{S}(\alpha) = \sum_{\bm{x} \in \mathcal{O}^n } \psi \left( \alpha F(M\bm{x}+\bm{b}) \right) \widetilde{w}\left(\frac{M\bm{x} + \bm{b}}{P}\right).
\]
Since $F$ is diagonal we may factorise $\widetilde{S}(\alpha)$ as
\[
\widetilde{S}(\alpha) = \prod_{i=1}^n \widetilde{T}_i(\alpha),
\]
where
\[
\widetilde{T}_i(\alpha) = \sum_{\substack{x \in \mathcal{O} \\ \lvert Mx + b_i - x_{0,i}\rvert < \lvert P \rvert \widehat{N}^{-1}}} \psi (\alpha F_i (Mx+b_i)^3).
\]
Note that our counting function $N(P,\widetilde{w})$ agrees with the function $\rho_{M,\bm{b}}(n)$ and $\widetilde{S}(\alpha)$ agrees with $T(\alpha)$ in~\cite[Chapter 4]{lee2011birch}. In order to show weak approximation for the variety $X = \mathbb{V}(F) \subset \mathbb{P}^{n-1}$, by the same argument as the one provided in Section 4.9 of~\cite{lee2011birch}, it is enough to show the following result.
\begin{theorem}
Suppose $\cha(K)>3$. Then there exists some $\delta > 0$ such that
\[
N(P,\widetilde{w}) = \lvert M \rvert^{-3} \mathfrak{S} \mathfrak{I} \lvert P \rvert^{n-3} + O(\lvert P \rvert^{n-3-\delta}),
\]
where $\mathfrak{S}$ and $\mathfrak{I}$ are the singular series and the singular integral respectively as defined in~\eqref{eq.sing_ser_WA} and~\eqref{eq.sing_int_WA}.
\end{theorem}
We tackle this using a traditional circle method argument.

 We define the major arcs to be the set $\mathcal{M} \subset \TT$ given by
\[
\mathcal{M} = \bigcup_{\substack{r \in \mathcal{O}\\ \lvert r \rvert < \lvert P \rvert^{1/2} \\ r \text{ monic}  }} \bigcup_{\substack{a \in \mathcal{O} \\ \lvert a \rvert < \lvert r \rvert \\ (a,q) = 1}} \left\{ \alpha \in \TT \colon \lvert r\alpha - a \rvert < H_F^{-1} \lvert M \rvert^{-3} \lvert r \rvert \lvert P \rvert^{-5/2} \right\},
\]
and we take the minor arcs to be the complement $\mathfrak{m} = \TT \setminus \mathcal{M}$. 

In this context, provided $\cha(K)>3$, Weyl's inequality~\cite[Lemma 4.3.6]{lee2011birch} tells us that
\[
|\widetilde{T}_i(\alpha)|\ll |P|^{1+\varepsilon}\left(\frac{|P|+|r|+|P|^3|r\alpha-a|}{|P|^3}+\frac{1}{|r|+|P|^3|r\alpha-a|} \right)^{1/4}
\]
for $i=1,\dots, n$ if $a,r\in\O$ are such that $|a|<|r|$, $r$ monic and $(a,r)=1$. Using~\eqref{Eq: DirichletDissection} and the definition of the minor arcs, a similar argument that handed us~\eqref{Eq:EstimateMinorWaringPtwise} gives
\begin{equation} \label{eq.weyl_WA}
\sup_{\alpha \in \mathfrak{m}} \left\lvert \widetilde{T}_i(\alpha)\right\rvert \ll \lvert P \rvert^{7/8+\varepsilon},
\end{equation}
for any $\varepsilon >0$.
We are now ready to finish our treatment of the minor arcs. If $n\geq 7$ we obtain
\[
\int_{\mathfrak{m}} \lvert \widetilde{S}(\alpha) \rvert \dd \alpha = \int_{\mathfrak{m}} \left\lvert \prod_{i=1}^n \widetilde{T}_i(\alpha) \right\rvert \dd\alpha \ll \sup_{\alpha \in \mathfrak{m}} \left\lvert \widetilde{T}_7(\alpha) \cdots \widetilde{T}_n(\alpha) \right\rvert \int_{\TT} \left\lvert \prod_{i=1}^6 \widetilde{T}_i(\alpha) \right\rvert \dd\alpha.
\] 
The integral can be dealt with as follows. By Hölder's inequality we find
\[
\int_{\mathfrak{m}} \left\lvert \prod_{i=1}^6 \widetilde{T}_i(\alpha) \right\rvert \dd\alpha \leq \prod_{i=1}^6\left(\int_{\TT}|\widetilde{T}_i(\alpha)|^6\dd\alpha\right)^{1/6}.
\]
Now the last quantity is equal to
\[
 \prod_{i=1}^6\#\left\{\bm{x}\in \O^6\colon x_j \equiv b_i \, \mathrm{mod} \, M, |x_j/P-x_{0,i}|<\widehat{N}^{-1}, \text{ for all $j$, } \sum_{j=1}^3 x_j^3 =\sum_{j=4}^6 x_j^3\right\}^{1/6},
\]
which in turn is bounded by 
\[
\prod_{i=1}^6 \#\{\bm{x}\in \O^6 \colon |\bm{x}|< |\bm{x}_0||P|, \, x_1^3+x_2^3+x_3^3=x_4^3+x_5^3+x_6^3\}^{1/6},
\]
if $\lvert P \rvert$ is sufficiently large. An application of Theorem~\ref{Th:TheTheorem} therefore yields
\[
 \int_{\TT} \left\lvert \prod_{i=1}^6 \widetilde{T}_i(\alpha) \right\rvert \dd\alpha \ll \lvert P \rvert^{3+\varepsilon}. 
\]
Once combined with~\eqref{eq.weyl_WA} we thus obtain
\[
\int_{\mathfrak{m}} \lvert \widetilde{S}(\alpha) \rvert \dd \alpha \ll \lvert P \rvert^{n-3-(n-7)/8+\varepsilon}
\]
for any $\varepsilon >0$, which is satisfactory if $n\geq 7$.
We now turn to the major arcs. Given $a,r \in \mathcal{O}$ write
\[
\widetilde{S}_r(a) \coloneqq \sum_{\lvert \bm{x} \rvert < \lvert r \rvert} \psi \left( \frac{a F(M \bm{x} + \bm{b})}{r} \right).
\]
For any $Y \in \mathbb{R}$ we define the truncated singular series
\[
\mathfrak{S}(\widehat{Y}) \coloneqq \sum_{\substack{\lvert r \rvert < \widehat{Y} \\ r \text{ monic}}} \sum_{\substack{\lvert a \rvert < \lvert r \rvert \\ (a,r) = 1}} \lvert r \rvert^{-n} \widetilde{S}_r(a),
\]
and the truncated singular integral to be
\[\
\mathfrak{I}(\widehat{Y}) = \int_{\lvert \gamma \rvert < H_F^{-1} \widehat{Y}} I(\gamma) \dd \gamma,
\]
where
\[
I(\gamma) = \int_{\TT^n} \psi(\gamma F(\bm{x})) \widetilde{w}(\bm{x}) \dd\bm{x}.
\]
Then from (4.6.30) in~\cite{lee2011birch} it follows that we have
\[
\int_{\mathcal{M}} \widetilde{S}(\alpha) \dd \alpha = \lvert M \rvert^{-3} \mathfrak{S}(\lvert P \rvert^{1/2}) \mathfrak{I}(\lvert P \rvert^{1/2}) \lvert P \rvert^{n-3}.
\]
It remains to study the convergence of the singular integral and singular series. 
In order to handle the singular series we will need upper bounds for $\widetilde{S}_r(a)$. First, we record the following multiplicative property, which is shown in~\cite[Lemma 4.7.2]{lee2011birch}. If $r_1,r_2 \in \mathcal{O}$ are coprime then
\[
\widetilde{S}_{r_1 r_2}(a) = \widetilde{S}_{r_1}(a_1) \widetilde{S}_{r_2}(a_2),
\]
 where $a_i \in \mathcal{O}$ are such that $a_1 \equiv a \tilde{r}_2 \mod r_1$ and $a_2 \equiv a \tilde{r}_1 \mod r_2$, where $\tilde{r}_1, \tilde{r}_2$ denote the multiplicative inverses modulo $r_2,r_1$, respectively. Thus, from~\eqref{Eq: HuaEstimate} in combination with the divisor estimate, it follows that we have
\begin{equation} \label{eq.WA_exp_sum_bound}
    \widetilde{S}_r(a) \ll \lvert r \rvert^{2n/3+\varepsilon},
\end{equation}
where the constant may depend on $M,b$ and $\varepsilon$.

Using this we see that
\[
 \sum_{ \substack{\lvert r \rvert = \widehat{Y} \\ r \text{ monic}}}\sum_{\substack{\lvert a \rvert < \lvert r \rvert \\ (a,r) = 1}} \lvert r \rvert^{-n} \left|\widetilde{S}_r(a)\right| \ll \widehat{Y}^{(2-n/3+\varepsilon )}.
\]
Since $n \geq 7$ we deduce absolute convergence of the series
\begin{equation} \label{eq.sing_ser_WA}
    \mathfrak{S} = \sum_{\substack{ r \text{ monic}}} \sum_{\substack{\lvert a \rvert < \lvert r \rvert \\ (a,r) = 1}} \lvert r \rvert^{-n} \widetilde{S}_r(a),
\end{equation}
which is the singular series. Moreover choosing positive $\varepsilon < (n-6)/6$ we find
\begin{equation} \label{eq.sing_ser_estimate}
    \mathfrak{S} - \mathfrak{S}(\lvert P \rvert^{1/2}) \ll \lvert P \rvert^{1-n/6+\varepsilon},
\end{equation}
if $n \geq 7$ upon redefining $\varepsilon$.
We turn to the singular integral.
Let $\bm{x}_0 \in K_\infty$ be a non-singular point of $X \subset \mathbb{P}^{n-1}$. In~\cite{CubicHypersurfacesBV} it is shown in Lemma 7.5 and the paragraphs preceding it that 
\[
\mathfrak{I}(\widehat{Y}) = \mathfrak{I}(\widehat{N}/\lvert \nabla F(\bm{x}_0)\rvert) = \frac{1}{\lvert \nabla F(\bm{x}_0) \rvert \widehat{N}^{n-1}}
\]
whenever $\widehat{Y} \geq \widehat{N}/\lvert \nabla F(\bm{x}_0) \rvert$. Thus clearly
$\lim_{\widehat{Y} \rightarrow \infty} \mathfrak{I}(\widehat{Y})$ exists and is equal to
\begin{equation} \label{eq.sing_int_WA}
    \mathfrak{I} \coloneqq \lim_{\widehat{Y} \rightarrow \infty} \mathfrak{I}(\widehat{Y}) = \frac{1}{\lvert \nabla F(\bm{x}_0) \rvert \widehat{N}^{n-1}}.
\end{equation}
We conclude that
\[
N(P,\widetilde{w}) = \lvert M \rvert^{-3} \mathfrak{S} \mathfrak{I} \lvert P \rvert^{n-3} + O(\lvert P \rvert^{n-3-1/8+\varepsilon}),
\]
as desired.
\section{Special solutions and the case $n=4$}\label{Se:SpecialSolutions}
In this section we will concern ourselves with understanding how the special solutions of $F^*(\bm{c}) = 0$ in the case $n=4$ relate to the solutions of $F(\bm{x}) = 0$ on rational lines. The goal of this section is to prove the following lemma, from which Theorem~\ref{Th:TheoTheorem.n=4} immediately follows.
\begin{lemma} \label{lem.special_soln=lines}
	For any $\varepsilon > 0$ the following holds
	\begin{equation} \label{eq.special_solutions=lines}
		|P|^4\sum_{\substack{r \; \mathrm{monic}\\ |r|\leq \widehat{Q}}}|r|^{-4}\int_{|\theta|<|r|^{-1}\widehat{Q}^{-1}}\sum_{\bm{c}} {\vphantom{\sum}}^{\mathrm{spec}}S_r(\bm{c})I_r(\theta,\bm{c})\dd\theta = \sum_{\bm{x}}^{}{\vphantom{\sum}}^{\mathrm{line}} w(P^{-1}\bm{x}) + O(\lvert P \rvert^{3/2+\varepsilon}),
	\end{equation}
	where $\sum_{\bm{c}}^{\mathrm{spec}}$ denotes the sum over the \emph{special solutions} $\bm{c} \in \O^4\setminus\{\bm{0}\}$ of $F^*(\bm{c}) =0$ such that 
	\begin{equation} \label{eq.condition_c_special_special}
	(F_1^{-1}c_1^3)^{1/2} \pm (F_2^{-1}c_2^3)^{1/2} = (F_3^{-1}c_3^3)^{1/2} \pm (F_4^{-1}c_4^3)^{1/2} = 0
	\end{equation}
	and $\sum_{\bm{x}}^{\mathrm{line}}$ denotes the sum over points $\bm{x} \in \O^4$ satisfying
	\begin{equation} \label{eq.ponits_lines_condition}
	F_1 x_1^3 + F_2 x_2^3 = F_3x_3^3 + F_4x_4^3 = 0.
	\end{equation}
\end{lemma}
For notational convenience, this lemma only considers the case of lines such that $(i,j,k,l) = (1,2,3,4)$ in the language of Theorem~\ref{Th:TheoTheorem.n=4}. By the symmetry of the situation at hand it is clear that the result follows for any permutation of indices.

\subsection{Analysis of special solutions}
We begin by noting that with an error of $O(\lvert P \rvert^{3/2+\varepsilon})$ we may include tuples $\bm{c} \in \O^4\setminus\{\bm{0}\}$ satisfying~\eqref{eq.condition_c_special_special} such that $c_i = 0$ for at least one $i$ in the sum appearing in the left hand side of~\eqref{eq.special_solutions=lines}. Write $\sum_{\bm{c}}^{\widetilde{\mathrm{spec}}}$ for the sum over such tuples $\bm{c}$. Note for such $\bm{c}$ Lemma~\ref{lem.integral_estimate_for_I_r} gives
\[
I_r(\bm{c}) \ll \lvert P \rvert^{-5/2} \lvert \bm{c} \rvert^{-1},
\]
for any $r \in \O$. Also note that $I_r(\theta, \bm{c}) = 0$ if $\lvert \bm{c} \rvert \gg \lvert P \rvert^{1/2}$. From~\eqref{Eq: DeligneS_r(c)} and Lemma~\ref{Le: ExpSumAverageSquarefull}, where we apply the second part with $m=0$, we obtain
\[
S_r(\bm{c}) \ll \lvert r \rvert^\varepsilon \lvert r_1 \rvert^{3} \lvert r_2\rvert^{4-1/3},
\]
where $r_1$ denotes the cube-free and $r_2$ the cube-full part of $r$. Hence
\[
\sum_{\substack{ r \text{ monic} \\ \lvert r \rvert \leq \widehat{Q}}} \lvert r \rvert^{-4} S_r(\bm{c}) \ll \lvert P \rvert^\varepsilon \left( \sum_{\lvert r_1 \rvert \leq \widehat{Q}} \lvert r_1 \rvert^{-1} \right) \left( \sum_{\lvert r_2 \rvert \leq \widehat{Q}} \lvert r_2 \rvert^{-1/3} \right) \ll \lvert P \rvert^\varepsilon,
\]
since the number of cube-full $r_2$ of a fixed absolute value of $\widehat{Y}$, say, is at most $P(\widehat{Y}^{1/3})$. To summarise, we found that the contribution to the left hand side of~\eqref{eq.special_solutions=lines} is at most
\[
|P|^4\sum_{\substack{r\text{ monic}\\ |r|\leq \widehat{Q}}}|r|^{-4}\sum_{\bm{c}} {\vphantom{\sum}}^{\widetilde{\mathrm{spec}}}S_r(\bm{c})I_r(\bm{c})\ll \lvert P \rvert^{3/2+\varepsilon} \sum_{0 < \lvert \bm{c}\rvert \leq \lvert P \rvert^{1/2}} \hspace{-16pt}{\vphantom{\sum}}^{\widetilde{\mathrm{spec}}} \lvert \bm{c} \rvert^{-1} \ll \lvert P \rvert^{3/2+\varepsilon},
\]
where the last estimate follows since there are only $O(\widehat{C})$ vectors $\bm{c}$ of absolute value $\widehat{C}$, say, appearing in $\sum_{\bm{c}}^{\widetilde{\mathrm{spec}}}$. 

We may assume that both $F_1/F_2$ and $F_3/F_4$ are cubes in $K$. Otherwise the conclusion of the lemma is easily seen to be true, since there are no special solutions and $O(\lvert P \rvert)$ points $\bm{x}$ satisfying~\eqref{eq.ponits_lines_condition}. Therefore there exist at most $O(1)$ many different possible $\rho_i \in \O$ with $(\rho_1,\rho_2) = (\rho_3, \rho_4) = 1$  and $\lambda, \mu \in \O$ such that
\[
F_1 = \lambda \rho_1^3, \quad F_2 = \lambda \rho_2^3, \quad F_3 = \mu \rho_3^3, \quad F_4 = \mu \rho_4^3.
\]
The different possibilites for $\rho_i$ come from the potential existence of non-trivial third roots of unity in $K$. For a choice of $\rho_i \in \O$ if we write
\[
c_1 = \rho_1 d_1, \quad c_2 = \rho_2 d_1, \quad c_3 = \rho_3 d_2, \quad c_4 = \rho_4 d_2,
\]
then as we run through the possible choices of $\rho_i$ and as $\bm{d}$ runs through $\O^2$, then $\bm{c}$ runs through solutions of $F^{*}(\bm{c}) = 0$ satisfying~\eqref{eq.condition_c_special_special} .
Given a choice of $\rho_i$ there exist $\rho_i' \in \O$ such that
\[
\rho_1 \rho_2' - \rho_2 \rho_1' = \rho_3 \rho_4' - \rho_4 \rho_3' = 1.
\]
Then the change of variables $(x_1, x_2, x_3, x_4) \mapsto (y_1,y_2,z_1,z_2)$ given by
\[
\begin{pmatrix}
	y_1 \\
	z_1 \\
	y_2 \\
	z_2
\end{pmatrix} =
\begin{pmatrix}
	\rho_1 & \rho_2 & 0 & 0 \\
	\rho_1' & \rho_2' & 0 & 0 \\
	0 & 0 & \rho_3 & \rho_4 \\
	0 & 0 & \rho_3' & \rho_4'
\end{pmatrix}
\begin{pmatrix}
	x_1 \\
	x_2 \\
	x_3 \\
	x_4
\end{pmatrix}
\]
is unimodular. Moreover the inverse of this is easily seen to be
\[
\begin{pmatrix}
	x_1 \\
	x_2 \\
	x_3 \\
	x_4
\end{pmatrix} = 
\begin{pmatrix}
	\rho_2' & -\rho_2 & 0 & 0 \\
	-\rho_1' & \rho_1 & 0 & 0 \\
	0 & 0 & \rho_4' & -\rho_4 \\
	0 & 0 & -\rho_3' & \rho_3
\end{pmatrix}
\begin{pmatrix}
	y_1 \\
	z_1 \\
	y_2 \\
	z_2
\end{pmatrix}.
\] 
We will write $\bm{x}(\bm{y},\bm{z})$ for $\bm{x}$ arising from this linear transformation.
An easy calculation reveals 
\[
F(\bm{x}(\bm{y},\bm{z})) = y_1 Q_1(y_1,z_1) + y_2 Q_2(y_2,z_2) =: \widetilde{F}(\bm{y},\bm{z}),
\]
where $Q_i$ are the quadratic forms given by
\[
Q_1(y,z) = \frac{\lambda}{4} \left( y^2 + 3 \{ 2\rho_1 \rho_2z - (\rho_1 \rho_2' + \rho_1' \rho_2)y \}^2 \right),
\]
and
\[
Q_2(y,z) = \frac{\mu}{4} \left( y^2 + 3 \{ 2\rho_3 \rho_4z - (\rho_3 \rho_4' + \rho_3' \rho_4)y \}^2 \right).
\]
With this notation we then find 
\[
S_r(\bm{c}) = \sideset{}{'}\sum_{\lvert a \rvert < \lvert r \rvert} \sum_{\lvert \bm{g} \rvert, \lvert \bm{h} \rvert < \lvert r \rvert} \psi \left( \frac{a \widetilde{F}(\bm{g}, \bm{h}) +  \bm{g} \cdot \bm{d}}{r} \right),
\]
and
\[
I_r(\theta, \bm{c}) = \int_{K_\infty^2} \int_{K_\infty^2} w(\bm{x}(\bm{y},\bm{z})) \psi\left(\theta P^3 \widetilde{F}(\bm{y},\bm{z}) +P \frac{\bm{y} \cdot \bm{d}}{r} \right) \dd \bm{y} \dd \bm{z}.
\]
We make the change of variables $\bm{y} = P^{-1}(\bm{g}+r \bm{v})$ in the integral to obtain 
\begin{multline*}
 I_r(\theta, \bm{c}) = \lvert r \rvert^{2} \lvert P \rvert^{-2} \int_{K_\infty^2} \int_{K_\infty^2} w(\bm{x}(P^{-1}(\bm{g}+r \bm{v}),\bm{z})) \\
 \times \psi \left( \theta P^3 \widetilde{F} (P^{-1}(\bm{g}+r\bm{v}),\bm{z}) + \frac{\bm{g} \cdot \bm{d}}{r} \right) \psi(\bm{v} \cdot \bm{d}) \dd \bm{v} \dd \bm{z}.    
\end{multline*}
Hence we find
\[
\sum_{\bm{c}} {\vphantom{\sum}}^{\mathrm{spec}} S_r(\bm{c}) I_r(\theta, \bm{c}) = \lvert r \rvert^2 \lvert P \rvert^{-2} \sum_{\rho_i} \sum_{\lvert \bm{g} \rvert < \lvert r \rvert} \int_{K_\infty^2} \sum_{\bm{d} \in \O^2} \int_{K_\infty^2} f_{\bm{g},\bm{z}}(\theta, \bm{v}) \psi(\bm{v} \cdot \bm{d}) \dd \bm{v} \dd \bm{z},
\]
where $\sum_{\rho_i}$ sums over the finitely many possible choices for $\rho_i \in \mathcal{O}$ as above and where
\[
f_{\bm{g},\bm{z}}(\theta, \bm{v}) = \sideset{}{'}\sum_{\lvert a \rvert < \lvert r \rvert} \sum_{\lvert \bm{h} \rvert < \lvert r \rvert} w(\bm{x}(P^{-1}(\bm{g}+r \bm{v}),\bm{z})) \psi \left( \theta P^3 \widetilde{F} (P^{-1}(\bm{g}+r\bm{v}),\bm{z}) + \frac{a \widetilde{F}(\bm{g}, \bm{h})}{r} \right).
\]
Poisson summation~\eqref{eq.poisson summation} yields
\[
\sum_{\bm{d} \in \O^2} \int_{K_\infty^2} f_{\bm{g},\bm{z}}(\theta, \bm{v}) \psi(\bm{v} \cdot \bm{d}) \dd \bm{v} = \sum_{\bm{s} \in \O^2} f_{\bm{g},\bm{z}}(\theta, \bm{s}).
\] 
We make the change of variables $\bm{j} = \bm{g} + r\bm{s}$ and the substitution $\bm{z} = P^{-1} \bm{t}$ in order to obtain
\[
\sum_{\bm{c}} {\vphantom{\sum}}^{\mathrm{spec}} S_r(\bm{c}) I_r(\bm{c}) = \lvert r \rvert^2 \lvert P \rvert^{-4} \sum_{\rho_i} \sum_{\bm{j} \in \O^2} T_r(\bm{j}) J_r(\bm{j},\theta),
\]
where 
\[
T_r(\bm{j}) = \sideset{}{'}\sum_{\lvert a \rvert < \lvert r \rvert} \sum_{\lvert \bm{h} \rvert < \lvert r \rvert} \psi \left( \frac{a \widetilde{F}(\bm{j},\bm{h})}{r} \right),
\]
and
\[
J_r(\bm{j},\theta) = \int_{K_\infty^2}w(P^{-1}\bm{x}(\bm{j},\bm{t})) \psi(\theta \widetilde{F}(\bm{j}, \bm{t})) \dd \bm{t}.
\]
Further we will write
\[
J_r(\bm{j}) \coloneqq \int_{\lvert \theta \rvert < \lvert r \rvert^{-1} \widehat{Q}^{-1}} J_r(\bm{j},\theta) \dd \theta.
\]
We can summarise our findings until now as follows.
\begin{lemma} \label{lem.special_solutions_sum_first_step}
	We have
	\begin{multline} \label{eq.special_solutions_sum_first_step}
		|P|^4\sum_{\substack{r \; \mathrm{monic}\\ |r|\leq \widehat{Q}}}|r|^{-4}\sum_{\bm{c}} {\vphantom{\sum}}^{\mathrm{spec}}S_r(\bm{c})I_r(\bm{c}) 
		 = \sum_{\rho_i} \sum_{\substack{r \; \mathrm{monic}\\ |r|\leq \widehat{Q}}}|r|^{-2} \sum_{\bm{j} \in \O^2} T_r(\bm{j}) J_r(\bm{j}) + O(\lvert P \rvert^{3/2+\varepsilon}).
	\end{multline}
\end{lemma}
We now follow a strategy that is very similar to the usual delta method. The main term will come from $\bm{j} = \bm{0}$ and it then remains to estimate $T_r(\bm{j})$ and $J_r(\bm{j},\theta)$ for $\bm{j} \neq \bm{0}$.

\subsection{The main term}
\begin{lemma}
    For all $P \in \O \setminus \{ 0 \}$ we have
    \begin{equation*}
        \sum_{\rho_i} \sum_{\substack{r \; \mathrm{monic}\\ |r|\leq \widehat{Q}}}|r|^{-2} T_r(\bm{0}) J_r(\bm{0})  = \sum_{\bm{x}}^{}{\vphantom{\sum}}^{\mathrm{line}} w(P^{-1} \bm{x}) + O(1).
    \end{equation*}
\end{lemma}

\begin{proof}
Since $\widetilde{F}(\bm{0},\bm{z}) = 0$ for all $\bm{z} \in K_\infty^2$ we have 
\[
T_r(\bm{0}) =  \sideset{}{'}\sum_{\lvert a \rvert < \lvert r \rvert}\lvert r \rvert^2,
\]
and 
\[
J_r(\bm{0},\theta) = \int_{K_\infty^2} w(P^{-1}\bm{x}(\bm{0},\bm{t})) \dd \bm{t}.
\]
Therefore, the term arising from $\bm{j} = \bm{0}$ on the right hand side of~\eqref{eq.special_solutions_sum_first_step} is equal to
\[
\sum_{\rho_i} \int_{K_\infty^2} w(P^{-1}\bm{x}(\bm{0},\bm{t})) \dd \bm{t}  \sum_{\substack{r \; \mathrm{monic}\\ |r|\leq \widehat{Q}}}  \sideset{}{'}\sum_{\lvert a \rvert < \lvert r \rvert} \int_{|\theta|<|r|^{-1}\widehat{Q}^{-1}} \dd \theta.
\]
But from Dirichlet's approximation theorem~\eqref{Eq: DirichletDissection} we see
\[
\sum_{\substack{r \; \mathrm{monic}\\ |r|\leq \widehat{Q}}}  \sideset{}{'}\sum_{\lvert a \rvert < \lvert r \rvert} \int_{|\theta|<|r|^{-1}\widehat{Q}^{-1}} \dd \theta = \mu \left( \mathbb{T} \right) = 1.
\]

Further, it is easily seen that
\[
\sum_{\bm{x}}^{}{\vphantom{\sum}}^{\mathrm{line}} w(P^{-1} \bm{x}) = \sum_{\rho_i} \sum_{\bm{z} \in \O^2} w(P^{-1}\bm{x}(\bm{0},\bm{z})).
\]
But since $K_\infty^2 = \bigsqcup_{\bm{z} \in \O^2} (\bm{z} + \mathbb{T})$ we have 
\[
\int_{K_\infty^2}  w(P^{-1}\bm{x}(\bm{0},\bm{t}))  \dd \bm{t} = \sum_{\bm{z} \in \O^2} \int_{\TT^2} w(P^{-1}\bm{x}(\bm{0},\bm{z}+\bm{\alpha})) \dd \bm{\alpha}.
\]
If $\bm{z} \in \O \setminus \{ \bm{0} \}$ then $\lvert \bm{x}(\bm{0},\bm{z} + \bm{\alpha}) \rvert = \lvert \bm{x}(\bm{0},\bm{z}) \rvert$ for all $\bm{\alpha} \in \TT^2$ and so 
\[
\int_{\TT^2} w(P^{-1}\bm{x}(\bm{0},\bm{z}+\bm{\alpha})) \dd \bm{\alpha} = w(P^{-1}\bm{x}(\bm{0},\bm{z}))
\]
for such $\bm{z}$. We also clearly have $\int_{\TT^2} w(P^{-1}\bm{x}(\bm{0},\bm{\alpha})) \dd \bm{\alpha} \ll 1$ and so
\[
\int_{K_\infty^2} w(P^{-1}\bm{x}(\bm{0},\bm{t})) \dd \bm{t} =  \sum_{\bm{z} \in \O^2} w(P^{-1}\bm{x}(\bm{0},\bm{z})) + O(1),
\]
whence the Lemma follows.
\end{proof}

\subsection{Estimating the error term}
In this section we make a choice of $\rho_1, \hdots, \rho_4$ and  bound the contribution made from terms such that $\bm{j} \neq \bm{0}$. Once we showed the desired bound for a particular choice, Lemma~\ref{lem.special_soln=lines} will follow since there are only $O(1)$ different possibilities for $\rho_i$.

We begin by bounding $J_r(\bm{j})$
in the case where $\bm{j} \neq \bm{0}$. Note first that $w(P^{-1} (\bm{x}(\bm{j},\bm{t}))) = 0 $ if $\bm{j} \gg \lvert P \rvert$ and so $J_r(\bm{j}) = 0$ if $\bm{j} \gg \lvert P \rvert$. Further this allows us to exchange the integral over $\theta$ with the sum over $\bm{j}$ in~\eqref{eq.special_solutions_sum_first_step}. Note further from~\eqref{Eq: OrthogonalityOfCharacters} that we have
\[
\int_{\lvert \theta \rvert < \lvert r \rvert^{-1} \widehat{Q}^{-1}} \psi(\theta \widetilde{F}(\bm{j},\bm{t})) \dd \theta = \begin{cases}
\lvert r \rvert^{-1} \widehat{Q}^{-1}, \quad &\text{if $\lvert \widetilde{F}(\bm{j},\bm{t})\rvert <\lvert r \rvert \widehat{Q}$} \\
0, &\text{otherwise.}
\end{cases}
\]
Thus we find
\[
J_r(\bm{j}) \ll \mu(\bm{j},r) \lvert r\rvert^{-1} \widehat{Q}^{-1},
\]
where
\[
\mu(\bm{j},r) = \mathrm{meas}\left( \left\{ \bm{t} \in K_\infty^2 \colon \lvert \bm{t} \rvert \ll \lvert P \rvert, \; \lvert \widetilde{F}(\bm{j},\bm{t}) \rvert <\lvert r \rvert \widehat{Q} \right\}\right).
\]
To estimate this measure we simplify the expressions involved by making the substitution
\[
u_1 = 2 \rho_1 \rho_2 t_1 - (\rho_1 \rho_2' + \rho_1' \rho_2)j_1, \quad u_2 = 2 \rho_3 \rho_4 t_2 - (\rho_3 \rho_4' + \rho_3' \rho_4)j_1.
\]
After this linear change of variables $\widetilde{F}$ takes the form 
\[
\widetilde{G}(\bm{j},\bm{u}) = \lambda j_1(3u_1^2+j_1^2) + \mu j_2 (3u_2^2+j_2^2).
\]
Since the change of variables is linear of constant, non-vanishing Jacobian it is sufficient to consider
\[
\mu_{\widetilde{G}}(\bm{j}, r) \coloneqq  \mathrm{meas} \left(\left\{ \bm{u} \in K_\infty^2 \colon \lvert \bm{u} \rvert \ll \lvert P \rvert, \; \lvert \widetilde{G}(\bm{j},\bm{u}) \rvert <\lvert r \rvert \widehat{Q} \right\}\right).
\]
If $j_2=0$ then using Lemma~\ref{lem.measure_parabola} it is easily seen that
\[
\mu_{\widetilde{G}}(\bm{j},r) \ll \lvert P \rvert \left(\frac{\lvert r \rvert \widehat{Q}}{\lvert j_1 \rvert}\right)^{1/2},
\]
and similarly if $j_1=0$. So assume $j_1j_2 \neq 0$.
In  this case, note that we have
\[
\mu_{\widetilde{G}}(\bm{u}, r) \ll \sum_{k,m = -\infty}^{\log_q\lvert P \rvert} \sum_{\substack{U_1 = q^k \\ U_2 = q^m}} \mu_{\widetilde{G}}(\bm{j},r,U_1,U_2),
\]
where
\[
\mu_{\widetilde{G}}(\bm{j},r,U_1,U_2) = \mathrm{meas}\left( \left\{ \bm{u} \in K_\infty^2 \colon  \lvert u_1 \rvert = U_1, \;\lvert u_2 \rvert = U_2, \; \left\lvert \widetilde{G}(\bm{j},\bm{u}) \right\rvert <\lvert r \rvert \widehat{Q} \right\}\right).
\]
In the case where $U_1$ or $U_2 < \lvert P \rvert^{-1}$ we can use the trivial bound $O(U_1  U_2)$ for $\mu_{\widetilde{G}}(\bm{j},r,U_1,U_2)$ to deduce that the total contribution arising from such $U_1, U_2$ is bounded by $O(1)$. For the remaining contribution note if $\bm{u}$ satisfies $\widetilde{G}(\bm{j},\bm{u}) = 0$ then $u_1^2 = A + O(\lvert r \rvert \widehat{Q}/\lvert j_1 \rvert)$ for some function $A(j_1,j_2,u_2)$ and thus $u_1$ lies in a subset of measure $O(\lvert r \rvert \widehat{Q}/(U_1 \lvert j_1 \rvert))$. Therefore $\mu_{\widetilde{G}}(\bm{j},r,U_1,U_2) \ll U_2 \lvert r \rvert \widehat{Q}/(U_1 \lvert j_1 \rvert)$. Similarly, $\mu_{\widetilde{G}}(\bm{j},r,U_1,U_2) \ll U_1 \lvert r \rvert \widehat{Q}/(U_2 \lvert j_2 \rvert)$. Putting this together yields
\[
\mu_{\widetilde{G}}(\bm{j},r,U_1,U_2) \ll  \lvert r \rvert \widehat{Q} \lvert j_1 j_2 \rvert^{-1/2}.
\]
Since there are $\lvert P \rvert^\varepsilon$ pairs $U_1,U_2$ such that $\lvert P \rvert^{-1} \leq U_1, U_2 \leq \lvert P \rvert$ we deduce 
\[
\mu(\bm{j},r) \ll 1+ \lvert P \rvert^\varepsilon \lvert r \rvert \widehat{Q} \lvert j_1 j_2 \rvert^{-1/2}.
\]
We summarise our observations in the following lemma.
\begin{lemma} \label{lem.n=4_integral_bounds}
    Let $\bm{j} \in \O^2 \setminus \{ \bm{0} \}$ be such that $\lvert \bm{j} \rvert \ll \lvert P \rvert$. If $j_1j_2 \neq 0$, then we have
    \begin{equation} \label{eq.n=4_int_bound_1}
        J_r(\bm{j}) \ll \lvert P \rvert^\varepsilon \lvert j_1 j_2\rvert^{-1/2}.
    \end{equation}
    If $j_2 = 0$, then we have
    \begin{equation}
        J_r(\bm{j}) \ll \frac{\lvert P \rvert^{1/4}}{\left( \lvert j_1 \rvert \lvert r \rvert \right)^{1/2} }.
    \end{equation}
\end{lemma}
Next, we turn to estimating the exponential sums $T_r(\bm{j})$. Via the Chinese remainder theorem we have for all $r_1,r_2 \in \O$ such that $(r_1,r_2) = 1$ that
\begin{equation}\label{Eq: FactorT_r(j)}
T_{r_1 r_2}(\bm{j}) = T_{r_1}(\bm{j}) T_{r_2}(\bm{j}).
\end{equation}
Thus we may put our focus on $T_{r}(\bm{j})$ where $r = \varpi^k$ for irreducible $\varpi \in \O$. Note that
\[
\left\lvert \sum_{\lvert \bm{h} \rvert < \lvert r \rvert} \psi \left( \frac{a \widetilde{F}(\bm{j},\bm{h})}{r} \right) \right\rvert \leq \left\lvert \sum_{\lvert h_1 \rvert < \lvert r \rvert} \psi \left( \frac{a j_1 Q_1(j_1,h_1)}{r} \right) \right\rvert \, \left\lvert \sum_{\lvert h_1 \rvert < \lvert r \rvert} \psi \left( \frac{a j_2 Q_2(j_2,h_2)}{r} \right)  \right\rvert.
\]
A simple Weyl differencing type of argument further yields
\begin{align*}
    \left\lvert \sum_{\lvert h_1 \rvert < \lvert r \rvert} \psi \left( \frac{a j_1 Q_1(j_1,h_1)}{r} \right) \right\rvert^2 &=  \sum_{\lvert h \rvert, \lvert h_1 \rvert < \lvert r \rvert} \psi \left( \frac{a j_1 (Q_1(j_1,h+h_1) - Q_1(j_1,h_1))}{r} \right) \\
    &\ll \sum_{\lvert h \rvert < \lvert r \rvert} \left\lvert \sum_{\lvert h_1 \rvert < \lvert r \rvert} \psi \left( \frac{6a \lambda j_1 \rho_1^2 \rho_2^2 j_1 h_1 h }{r} \right) \right\rvert \\
    &= \lvert r \rvert \, \# \{ h \in \O \colon \lvert h \rvert < \lvert r \rvert, r \mid 6a \lambda j_1 \rho_1^2 \rho_2^2 j_1 h \} \\
    &\ll \lvert r \rvert \,\lvert (r,6a \lambda j_1 \rho_1^2 \rho_2^2 j_1 h) \rvert \\
    &\ll \lvert r \rvert \, \lvert (r,j_1) \rvert.
\end{align*}
We can find a similar estimate for the sum over $h_2$, which gives 
\[
T_r(\bm{j}) \ll \lvert r \rvert^2  \lvert (r,j_1) \rvert^{1/2} \lvert (r,j_2) \rvert^{1/2}.
\]
This will be sufficient for our purposes if $r$ is cube-full. However, for $r = \varpi$ or $r= \varpi^2$ we can do better. We begin by considering the case when $r = \varpi$ and we will further assume $\varpi \nmid (j_1,j_2)$. Note first that
\[
 \sideset{}{'}\sum_{\lvert a \rvert < \lvert \varpi \rvert}  \psi \left( \frac{a \widetilde{F}(\bm{j},\bm{h})}{\varpi} \right) = \sum_{\substack{\lvert a \rvert < \lvert \varpi \rvert \\ a \neq 0}}  \psi \left( \frac{a \widetilde{F}(\bm{j},\bm{h})}{\varpi} \right) = \begin{cases}
 \lvert \varpi \rvert -1, \quad&\text{if $\varpi \mid \widetilde{F}(\bm{j},\bm{h}$}), \\
 -1, &\text{otherwise.}
 \end{cases}
\]
Therefore we get
\begin{align*}
    T_\varpi(\bm{j}) &= (\lvert \varpi \rvert-1) \# \left\{ \lvert \bm{h} \rvert < \lvert \varpi \rvert \colon \varpi \mid \widetilde{F}(\bm{j},\bm{h})  \right\} - \# \left\{ \lvert \bm{h} \rvert < \lvert \varpi \rvert \colon \varpi \nmid \widetilde{F}(\bm{j},\bm{h})  \right\} \\
    &= \lvert \varpi \rvert \# \left\{ \lvert \bm{h} \rvert < \lvert \varpi \rvert \colon \varpi \mid \widetilde{F}(\bm{j},\bm{h})  \right\} - \lvert \varpi \rvert^2.
\end{align*}
The equation  $\widetilde{F}(\bm{j},\bm{h}) \equiv 0 \mod \varpi$ may be regarded as $Q(h_1,h_2,1)$ for a ternary quadratic form $Q(x,y,z)$. The quadratic form $Q$ is non-singular in $\O/\varpi$ if $\varpi\nmid j_1j_2 F_0(\bm{j})$, where $F_0(\bm{j}) = \lambda j_1^3 + \mu j_2^3$.
Since $\varpi$ is irreducible we have $\O/\varpi \cong \mathbb{F}_{\lvert \varpi \rvert}$ and so if $\varpi \nmid j_1j_2 F_0(\bm{j})$ then Theorem 6.26 in~\cite{lidl1997finite} gives
\[
\# \left\{ \lvert \bm{h} \rvert < \lvert \varpi \rvert \colon \varpi \mid \widetilde{F}(\bm{j},\bm{h})  \right\} = \lvert \varpi \rvert + O(1).
\]
We deduce $T_\varpi(\bm{j}) \ll \lvert \varpi \rvert$ in this case. Since $\varpi \nmid (j_1,j_2)$ the form $Q$ does not vanish identically in $\O/\varpi$ and so we have  
\[
\# \left\{ \lvert \bm{h} \rvert < \lvert \varpi \rvert \colon \varpi \mid \widetilde{F}(\bm{j},\bm{h})  \right\} \ll \lvert \varpi \rvert,
\]
whence $T_\varpi(\bm{j}) \ll \lvert \varpi \rvert^2$ if $\varpi \mid j_1j_2 F_0(\bm{j})$.

We now turn to analysing $T_{\varpi^2}(\bm{j})$. We assume $\varpi \nmid \lambda \mu \prod_{i=1}^5 \rho_i$. This condition affects only finitely many primes $\varpi$ and so the estimates that we obtain under this condition hold in general by adjusting the resulting constant. Put
\[
k_1 = 2 \rho_1 \rho_2 h_1 - (\rho_1 \rho_2' + \rho_1' \rho_2)j_1, \quad \text{and} \quad k_2 = 2 \rho_3 \rho_4 h_2 - (\rho_3 \rho_4' + \rho_3' \rho_4)j_2,
\]
so that after this change of variables we have
\[
\widetilde{F}(\bm{j},\bm{k}(\bm{h})) = \frac{1}{4}F_0(\bm{j}) + \frac{3}{4}(\lambda j_1 k_1^2 + \mu j_2 k_2^2).
\]
By our assumption on $\varpi$, as $\bm{h}$ ranges through values $\lvert \bm{h} \rvert < \lvert \varpi^2 \rvert$ we also have that $\bm{k}$ ranges through $\lvert \bm{k} \rvert < \lvert \varpi^2 \rvert$ under this change of variables. Hence we obtain
\[
T_{\varpi^2}(\bm{j}) = \sideset{}{'}\sum_{\lvert a \rvert < \lvert \varpi \rvert^2} \psi \left( \frac{a F_0(\bm{j})}{4\varpi^2} \right) \sum_{\lvert \bm{k} \rvert < \lvert \varpi \rvert^2} \psi \left( \frac{3a(\lambda j_1 k_1^2 + \mu j_2 k_2^2)}{4\varpi^2} \right).
\]
We can write $\bm{k} = \bm{u} + \varpi \bm{v}$ where $\lvert \bm{u} \rvert, \lvert \bm{v} \rvert < \lvert \varpi \rvert$. Then
\begin{align*}
   \sum_{\lvert k_i \rvert < \lvert \varpi \rvert^2} \psi \left( \frac{3a\lambda j_i k_i^2}{4\varpi^2} \right) &= \sum_{\lvert u_i \rvert < \lvert \varpi \rvert}  \psi \left( \frac{3a\lambda j_i u_i^2}{4\varpi^2} \right) \sum_{\lvert v_i \rvert < \lvert \varpi \rvert}  \psi \left( \frac{3a\lambda j_i u_i v_i}{4\varpi^2} \right) \\
   &= \lvert \varpi \rvert \sum_{\substack{\lvert u_i \rvert < \lvert \varpi \rvert \\ \varpi \mid j_i u_i}}  \psi \left( \frac{3a\lambda j_i u_i^2}{4\varpi^2} \right),
\end{align*}
for $i=1,2$ since $\varpi \nmid  a \lambda$. If $\varpi \nmid j_1j_2$ the above expression is just $\lvert \varpi \rvert$ and so we get in this case
\[
T_{\varpi^2}(\bm{j}) = \lvert \varpi \rvert^2 \sideset{}{'}\sum_{\lvert a \rvert < \lvert \varpi \rvert^2} \psi \left( \frac{a F_0(\bm{j})}{4\varpi^2} \right) = \begin{cases}
0, \quad &\text{if $\varpi \nmid F_0(\bm{j})$,} \\
-\lvert \varpi \rvert^3 &\text{if $\varpi \parallel F_0(\bm{j})$,}  \\
\lvert \varpi \rvert^4 - \lvert \varpi \rvert^3 &\text{if $\varpi^2 \mid F_0(\bm{j})$,}
\end{cases}
\]
and so in particular
\[
T_{\varpi^2}(\bm{j}) \ll \lvert \varpi \rvert^2 \lvert (\varpi^2,F_0(\bm{j})) \rvert.
\]
If, on the other hand, $\varpi \mid j_1$ we claim that $T_{\varpi^2}(\bm{j}) = 0$. Due to the standing assumption $\varpi \nmid (j_1,j_2)$ it follows that $\varpi \nmid j_2$ and thus the above gives
\[
T_{\varpi^2}(\bm{j}) = \lvert \varpi \rvert^2 \sum_{\lvert u_1 \rvert < \lvert \varpi \rvert} \sideset{}{'}\sum_{\lvert a \rvert < \lvert \varpi \rvert^2} \psi \left( \frac{a(F_0(\bm{j}) + 3\lambda j_1 u_1^2)}{4\varpi^2} \right).
\]
This vanishes unless $\varpi \mid F_0(\bm{j}) + 3\lambda j_1 u_1^2$. But since $\varpi \mid j_1$ this would imply $\varpi \mid \mu j_2^3$ and hence $\varpi \mid j_2$. As we excluded this case by assumption the claim follows. We summarise our analysis of $T_r(\bm{j})$ in a lemma.
\begin{lemma}\label{Eq: EstimatesT_r(j)}
    Let $\bm{j} \in \O^2 \setminus \{ \bm{0} \}$. Then we have
    \[
    T_r(\bm{j}) \ll \lvert r \rvert^2  \lvert (r,j_1) \rvert^{1/2} \lvert (r,j_2) \rvert^{1/2}
    \]
    for any $r \in \O \setminus\{0\}$. Further, if $r = \varpi$ or $r = \varpi^2$ for some irreducible $\varpi \in \O$ and if $\varpi \nmid (j_1,j_2)$ then we get
    \[
    T_r(\bm{j}) \ll \lvert r \rvert \lvert (r, j_1j_2 F_0(\bm{j})) \rvert.
    \]
\end{lemma}
We are now finally in a position to give a sufficiently good upper bound for the right hand side of~\eqref{eq.special_solutions_sum_first_step} and thus complete the proof of Theorem~\ref{Th:TheoTheorem.n=4}.
For this we fix a choice of $\rho_i$ and estimate the sum \[
\mathcal{S}\coloneqq \sum_{\substack{r\text{ monic}\\ |r|\leq \widehat{Q}}}|r|^{-2}\sum_{\substack{\bm{j}\in\O^2\\ |\bm{j}|\ll|P|}}T_r(\bm{j})J_r(\bm{j}).
\]
Since there are $O(1)$ possibilities for the $\rho_i$'s, this will be enough to show $\mathcal{S}\ll |P|^{3/2+\varepsilon}$.\\

We begin with the case when $j_1j_2F_0(\bm{j})\neq 0$. In this situation Lemma~\ref{lem.n=4_integral_bounds} yields 
\begin{equation}
    \mathcal{S}\ll |P|^{\varepsilon}\sum_{\bm{j}}|j_1j_2|^{-1/2}\sum_{\substack{r\text{ monic}\\ |r|\leq \widehat{Q}}}|r|^{-2}|T_r(\bm{j})|.
\end{equation}
Next we write $r=r_1r_2$ where $r_1$, $r_2$ monic are coprime,  and where $r_1$ is cube-free and $\varpi\mid r_1$ implies $\varpi\nmid(j_1,j_2)$. We can then factor $T_r(\bm{j})$ by~\eqref{Eq: FactorT_r(j)} to obtain 
\begin{align*}
    \mathcal{S}&\ll |P|^{\varepsilon}\sum_{\bm{j}}|j_1j_2|^{-1/2}\sum_{r_2}|r_2|^{-2}|T_{r_2}(\bm{j})|\sum_{r_1}|r_1|^{-2}|T_{r_1}(\bm{j})|\\
    &\ll |P|^{\varepsilon}\sum_{\bm{j}}|j_1j_2|^{-1/2}\sum_{r_2}|r_2|^{-2}|T_{r_2}(\bm{j})|\sum_{r_1}\frac{|(r_1,j_1j_2F_0(\bm{j}))|}{|r_1|},
\end{align*}
where we used Lemma~\ref{Eq: EstimatesT_r(j)} to estimate $T_{r_1}(\bm{j})$. For the inner sum we have 
\[
\sum_{r_1}\frac{|(r_1,j_1j_2F_0(\bm{j}))|}{|r_1|}\ll |P|^{\varepsilon}|j_1j_2F_0(\bm{j})|^{\varepsilon}\ll |P|^{2\varepsilon},
\]
since we assume $j_1j_2F_0(\bm{j})\neq 0$ and in general it holds $\widehat{Y}^{-1}\sum_{|r|= \widehat{Y}}|(G,r)|\ll (|G|\widehat{Y})^{\varepsilon}$ for any $Y\in\z_{\geq 0}$ and $G\in \O$.\\
Note that if $\varpi\parallel r_2$ or $\varpi^2\parallel r_2$, then $\varpi\mid (j_1,j_2)$. In particular, if we put $\eta(r_2)=\prod \varpi$, where the product is over all $\varpi\mid r_2$ such that $\varpi\parallel r_2$ or $\varpi^2\parallel r_2$, then we have $\bm{j}= \eta(r_2) \bm{k}$ for some $|\bm{k}|\ll |P|/|\eta(r_2)|$. It follows that 
\begin{align*}
\mathcal{S}&\ll |P|^\varepsilon\sum_{\substack{r\text{ monic}\\ |r|\leq \widehat{Q}}}|\eta(r)|^{-1}\sum_{\substack{|\bm{k}|\ll |P|/|\eta(r)|\\ k_1k_2\neq 0}} \frac{|(r,\eta(r)k_1)|^{1/2}|(r,\eta(r)k_2)|^{1/2}}{|k_1k_2|^{1/2}}\\
&\ll |P|^\varepsilon\sum_{\substack{r\text{ monic}\\ |r|\leq \widehat{Q}}}\sum_{\substack{|\bm{k}|\ll |P|/|\eta(r)|\\ k_1k_2\neq 0}} \frac{|(r,k_1)|^{1/2}|(r,k_2)|^{1/2}}{|k_1k_2|^{1/2}}.
\end{align*}
The sum over $\bm{k}$ above factors into $(\sum_k |(r,k)|^{1/2}|k|^{-1/2})^2$, which we can estimate as
\begin{align*}
    \sum_{\substack{|k|\ll |P|/|\eta(r)|\\ k\neq 0}}\frac{|(r,k)|^{1/2}}{|k|^{1/2}} &\ll \sum_{d\mid r}|d|^{1/2}\sum_{\substack{|k'|\ll |P|/|\eta(r)d|\\ (r, k')=1}}|k'd|^{-1/2}\\
    &\ll \sum_{d\mid r}|P|^{1/2}|\eta(r)|^{-1/2}.
\end{align*}
Since  $\sum_{d\mid r}1\ll  |r|^{\varepsilon}\ll |P|^\varepsilon$, we thus arrive at 
\[
\mathcal{S}\ll |P|^{1+\varepsilon}\sum_{|r|\leq \widehat{Q}} |\eta(r)|^{-1}.
\]
Next we write $r=st_1^2t_3$, where $s, t_1,t_3$ are pairwise coprime and monic, $t_3$ is cube-full and $s$ is square-free. With this notation we clearly have $\eta(r)=st_1$ and there are at most $O(\widehat{Q}^{1/3})=O(|P|^{1/2})$ available $t_3$, so that
\begin{align*}
    \mathcal{S}&\ll |P|^{3/2+\varepsilon}\sum_{|s|\leq \widehat{Q}}|s|^{-1}\sum_{|t_1|\leq (\widehat{Q}/|s|)^{1/2}}|t_1|^{-1}\\
    &\ll |P|^{3/2+\varepsilon}\sum_{|s|\leq \widehat{Q}}|s|^{-1}(\widehat{Q}/|s|)^{\varepsilon/2}\\
    &\ll|P|^{3/2+\varepsilon}\widehat{Q}^{3\varepsilon/2}.
\end{align*}
With a new choice of $\varepsilon$ this estimate suffices for our purpose.

Next we consider the case when $j_1j_2 F_0(\bm{j}) = 0$. If $j_1j_2 \neq 0$ but $F_0(\bm{j}) = 0$, then there exist some $j, \nu_i \in \O$ such that $j_i = \nu_i j$. The number of possible $\nu_i$ can be estimated by $O(1)$. In this case Lemma~\ref{lem.n=4_integral_bounds} and Lemma~\ref{Eq: EstimatesT_r(j)} yield
\[
J_r(\bm{j}) \ll \lvert P \rvert^{\varepsilon} \lvert j \rvert^{-1}, \quad \text{and} \quad T_r(\bm{j}) \ll \lvert r \rvert^2 \lvert (r,j) \rvert.
\]
The total contribution to $\mathcal{S}$ of such $\bm{j}$ is therefore bounded by
\[
\lvert P \rvert^\varepsilon \sum_{\substack{r\text{ monic}\\ |r|\leq \widehat{Q}}} \sum_{\substack{j \ll P \\ j \neq 0}} \lvert j \rvert^{-1} \lvert (r,j) \rvert \ll \lvert P \rvert^{3/2+\varepsilon},
\]
which is sufficient. 

Finally we need to consider the case when one of $j_i = 0$. We may assume $j_2 = 0$ since the other case is analogous. Write $j_1 = j$, then the second part of Lemma~\ref{lem.n=4_integral_bounds} gives
\[
J_r(\bm{j}) \ll \frac{\lvert P \rvert^{1/4}}{\left( \lvert j \rvert \lvert r \rvert \right)^{1/2} }.
\]
Combining the estimates in Lemma~\ref{Eq: EstimatesT_r(j)} also gives
\[
T_r(\bm{j}) \ll \lvert r \rvert^{5/2+\varepsilon} \lvert (j,r) \rvert m(r)^{-1/2},
\]
where $m(r) = \prod_{\varpi \parallel r} \varpi$.  The contribution to $\mathcal{S}$ of $\bm{j}$ under consideration is therefore bounded by
\[
\lvert P \rvert^{1/4} \sum_{\substack{r\text{ monic}\\ |r|\leq \widehat{Q}}} \sum_{\substack{j \ll P \\ j \neq 0}} \lvert (j,r) \rvert\lvert j\rvert^{-1/2}  m(r)^{-1/2}.
\]
Since $\sum_{0< j \ll P} \lvert (j,r) \rvert \lvert j \rvert^{-1/2} \ll q^\varepsilon \lvert P \rvert^{1/2 + \varepsilon}$ we get an overall bound 
\[
\lvert P \rvert^{3/4+\varepsilon} \sum_{\substack{r\text{ monic}\\ |r|\leq \widehat{Q}}} m(r)^{-1/2}.
\]
Write $r= r_1 r_2$ where $r_1$ is square-free and $r_2$ is square-full. Note that then $m(r) = r_1$ and there are at most $O\left(\left(\widehat{Q}/\lvert r_1 \rvert\right)^{1/2}\right)$ available $r_2$. Hence
\[
 \sum_{\substack{r\text{ monic}\\ |r|\leq \widehat{Q}}} m(r)^{-1/2} \ll \widehat{Q}^{1/2} \sum_{\substack{r_1\text{ monic}\\ |r_1|\leq \widehat{Q}}} \lvert r_1 \rvert^{-1} \ll \lvert P \rvert^{3/4+\varepsilon},
\]
and so the desired bound of $O(\lvert P \rvert^{3/2+\varepsilon})$ contributed from $\bm{j}$'s such that either $j_1 = 0$ or $j_2 = 0$ follows. Altogether, we have shown
\[
\mathcal{S} \ll \lvert P \rvert^{3/2 + \varepsilon},
\]
which completes the proof of Lemma~\ref{lem.special_soln=lines}.

\printbibliography
\end{document}